\documentclass[a4paper,12pt]{article}
\usepackage[cp1251]{inputenc}
\usepackage[russian]{babel}
\usepackage{amsfonts, amssymb, amsmath, amsthm, amscd}
\usepackage{cite}
\author{A.A. Vasil'eva}
\title{Widths of weighted Sobolev classes on a John domain\footnote{ The research was
supported by RFBR under grant no. 10-01-00442}}
\date{}
\begin{document}

\maketitle

\newenvironment{Biblio}{%
                  \renewcommand{\refname}{\footnotesize REFERENCES}%
                  \begin{thebibliography}{99}%
                  \itemsep -0.2em}%
                  {\end{thebibliography}}

\renewcommand{\le}{\leqslant}
\renewcommand{\ge}{\geqslant}
\newcommand{\sgn}{\mathrm {sgn}\,}
\newcommand{\inter}{\mathrm {int}\,}
\newcommand{\dist}{\mathrm {dist}}
\newcommand{\R}{\mathbb{R}}
\renewcommand{\C}{\mathbb{C}}
\newcommand{\Z}{\mathbb{Z}}
\newcommand{\N}{\mathbb{N}}
\newcommand{\Q}{\mathbb{Q}}
\theoremstyle{plain}
\newtheorem{Trm}{Theoerm}
\newtheorem{trma}{Theorem}
\newtheorem{Def}{Definition}
\newtheorem{Cor}{Corollary}
\newtheorem{Lem}{Lemma}
\newtheorem{Rem}{Remark}
\newtheorem{Sta}{Proposition}
\renewcommand{\proofname}{\bf Proof}
\renewcommand{\thetrma}{\Alph{trma}}
\section{Introduction}

Denote by $\N$, $\Z$, $\Z_+$, $\R$, $\R_+$ the sets of natural,
integer, nonnegative integer, real and nonnegative real numbers,
respectively.

Let $\Omega \subset \R^d$ be a bounded domain (a domain is an open
connected set), and let $g$, $v:\Omega\rightarrow \R_+$ be
measurable functions. For each measurable vector-valued function
$\varphi:\ \Omega\rightarrow \R^m$, $\varphi=(\varphi_k) _{1\le
k\le m}$, and for each $p\in [1, \, \infty]$ put
$$
\|\varphi\|_{L_p(\Omega)}= \Big\|\max _{1\le k\le m}|\varphi _k |
\Big\|_p.
$$
Let $\overline{\beta}=(\beta _1, \, \dots, \, \beta _d)\in
\Z_+^d:=(\N\cup\{0\})^d$, $|\overline{\beta}| =\beta _1+
\ldots+\beta _d$. For any distribution $f$ defined on $\Omega$ we
write $\displaystyle \nabla ^r\!f=\left(\partial^{r}\! f/\partial
x^{\overline{\beta}}\right)_{|\overline{\beta}| =r}$ (here partial
derivatives are taken in the sense of distributions) and denote by
$m_r$ the number of components of the vector-valued distribution
$\nabla ^r\!f$. We also write
$$
W^r_{p,g}(\Omega)=\left\{f:\ \Omega\rightarrow \R\big| \; \exists
\varphi :\ \Omega\rightarrow \R^{m_r}\!:\ \| \varphi \|
_{L_p(\Omega)}\le 1, \, \nabla ^r\! f=g\cdot \varphi\right\}
$$
\Big(we denote the corresponding function $\varphi$ by
$\displaystyle\frac{\nabla ^r\!f}{g}$\Big),
$$
\| f\|_{L_{q,v}(\Omega)}{=}\| f\|_{q,v}{=}\|
fv\|_{L_q(\Omega)},\qquad L_{q,v}(\Omega)=\left\{f:\Omega
\rightarrow \R| \; \ \| f\| _{q,v}<\infty\right\}.
$$

For $x\in \R^d$ and $\rho>0$ we shall denote by $B_\rho(x)$ a
closed euclidean ball of radius $\rho$ in $\R^d$
centered at the point $x$.
\begin{Def}
\label{fca} Let $\Omega\subset\R^d$ be a bounded domain, and let
$a>0$. We say that $\Omega \in {\bf FC}(a)$ if there exists a
point $x_*\in \Omega$ such that for any $x\in \Omega$ there exists
a curve $\gamma _x:[0, \, T(x)] \rightarrow\Omega$ with the
following properties:
\begin{enumerate}
\item $\gamma _x\in AC[0, \, T(x)]$, $|\dot \gamma _x|=1$ a.e.,
\item $\gamma _x(0)=x$, $\gamma _x(T(x))=x_*$,
\item $B_{at}(\gamma _x(t))\subset \Omega$ holds for any
$t\in [0, \, T(x)]$.
\end{enumerate}
\end{Def}

\begin{Def}
We say that $\Omega$ satisfies the John condition (and call
$\Omega$ a John domain) if $\Omega\in {\bf FC}(a)$ for some $a>0$.
\end{Def}
For a bounded domain the John condition is equivalent to the
flexible cone condition (see the definition in \cite{besov_il1}).

Reshetnyak in the papers \cite{resh1, resh2} constructed the integral
representation for functions defined on a John domain $\Omega$
in terms of their derivatives of order $r$. This integral representation
together with the Adams theorem on potentials
\cite{adams, adams1} yield that in the case $\frac rd-\left(\frac
1p-\frac 1q\right)_+>0$ the class $W^r_p(\Omega)$ is compactly embedded in
the space $L_q(\Omega)$ (i.e. the conditions of the compact embedding
are the same as for $\Omega=[0, \, 1]^d$). In this article it will be shown that
for a John domain such characteristics of embeddings of $W^r_p(\Omega)$ into $L_q(\Omega)$
as Kolmogorov and linear widths have the same order values as for $\Omega=[0, \, 1]^d$.

For properties of weighted Sobolev spaces and their generalizations,
see the books \cite{triebel, kufner, turesson, edm_trieb_book,
triebel1, edm_ev_book} and the survey paper \cite{kudr_nik}.
Sufficient conditions of boundedness and compactness for embeddings
of weighted Sobolev spaces into weighted $L_q$-spaces were obtained
by Kudryavtsev \cite{kudrjavcev}, Kufner
\cite{kufner}, Triebel \cite{triebel},
Lizorkin and Otelbaev \cite{liz_otel},
Gurka and Opic \cite{gur_opic}, Besov
\cite{besov1, besov2, besov3, besov4}, Antoci
\cite{antoci}, Gol’dshtein and Ukhlov
\cite{gold_ukhl}, and other authors.

Let $(X, \, \|\cdot\|_X)$ be a linear normed space, let
$n\in \Z_+$, ${\cal L}_n(X)$ be a the family of subspaces of $X$
whose dimension does not exceed $n$. Denote by $L(X, \,
Y)$ the space of continuous linear operators from $X$ into a normed space $Y$,
by ${\rm rk}\, A$ the dimension of the image of the operator
$A:X\rightarrow Y$, and by $\| A\| _{X\rightarrow Y}$
its norm. By the Kolmogorov $n$-width of a set
$M\subset X$ in the space $X$ we mean the quantity
$$d_n(M, \, X)=\inf _{L\in {\cal L}_n(X)} \sup_{x\in M}\inf_{y\in
L}\|x-y\|_X,$$ and by the linear $n$-width the quantity
$$
\lambda_n(M, \, X) =\inf_{A\in L(X, \, X), \, {\rm rk} A\le n}\sup
_{x\in M}\| x-Ax\| _X.
$$

The approximation numbers of an operator $A\in L(X, \, Y)$ are defined by
$$
{\cal A}_n(A)=\inf \{\| A-A_n\| _{X\rightarrow Y}:{\rm rk}\,
A_n\le n\}.
$$
If  $A$ is an embedding operator of a space $X$ in a space $Y$ and
if $M\subset X$ is a unit ball, then we write ${\cal A}_n(A)=
{\cal A}_n(M, \, Y)$. If the operator $A$ is compact, then from
Heinrich's result \cite{heinr} follows that
\begin{align}
\label{aneqln} {\cal A}_n(M, \, Y)=\lambda _n(A(M), \, Y).
\end{align}

Let $X$, $Y$ be sets, $f_1$, $f_2:\ X\times Y\rightarrow
\R_+$. We write $f_1(x, \, y)\underset{y}{\lesssim} f_2(x, \, y)$
(or $f_2(x, \, y)\underset{y}{\gtrsim} f_1(x, \, y)$) if for any
$y\in Y$ there exists $c(y)>0$ such that $f_1(x, \, y)\le
c(y)f_2(x, \, y)$ for each $x\in X$; $f_1(x, \,
y)\underset{y}{\asymp} f_2(x, \, y)$ if $f_1(x, \, y)
\underset{y}{\lesssim} f_2(x, \, y)$ and $f_2(x, \,
y)\underset{y}{\lesssim} f_1(x, \, y)$.

In the 1960–1970s authors investigated problems concerning the values
of the widths of function classes in $L_q$ (see \cite{bibl6,
tikh_babaj, busl_tikh, bib_ismag, bib_kashin, bib_majorov,
bib_makovoz, bibl9, bibl10, bibl11, bibl12, bibl13, kashin1,
kulanin} and also \cite{tikh_nvtp}, \cite{itogi_nt},
\cite{kniga_pinkusa}) and of finite-dimensional balls $B_p^n$ in
$l_q^n$. Here $l_q^n$ $(1\le q\le \infty)$ is the space
$\R^n$ with the norm $$\|(x_1, \, \dots , \, x_n)\| _q\equiv\|(x_1, \,
\dots , \, x_n)\| _{l_q^n}= \left\{
\begin{array}{l}(| x_1 | ^q+\dots+ | x_n | ^q)^{1/q},\text{ if
}q<\infty ,\\ \max \{| x_1 | , \, \dots, \, | x_n |\},\text{ if
}q=\infty ,\end{array}\right .$$ $B_p^n$ is the unit ball in
$l_p^n$. For $p\ge q$, Pietsch \cite{pietsch1} and Stesin
\cite{stesin} found the precise values of $d_n(B_p^\nu, \, l_q^\nu)$
and $\lambda_n(B_p^\nu, \, l_q^\nu)$. In the case of $p<q$,
Kashin \cite{bib_kashin}, Gluskin \cite{bib_gluskin} and
Garnaev, Gluskin \cite{garn_glus} determined order values
of the widths of finite-dimensional balls up to quantities
depending on $p$ and $q$ only.

Order estimates for widths of non-weighted Sobolev classes on a
segment were obtained by Tikhomirov, Ismagilov, Makovoz and Kashin
\cite{bibl6, tikh_babaj, bib_ismag, bib_kashin, bib_makovoz}. In
the case of multidimensional cube the upper estimate of widths
(which is not always precise) was first obtained by Birman and
Solomyak \cite{birm}. After publication of Kashin's result in
\cite{bib_kashin} estimates for widths of Sobolev classes on a
multidimensional torus and their generalizations were found by
Temlyakov and Galeev \cite{bibl9, bibl10, bibl11, bibl12, bibl13}.
In papers of Kashin \cite{kashin1} (for $d=1$) and Kulanin
\cite{kulanin} (for $d>1$) estimates of widths were found in the
case of ``small-order smoothness''. Here the upper estimate was
not precise in the case $d>1$ (with a logarithmic factor). The
correct estimate follows from embedding theorems between Sobolev
and Besov spaces and from the estimate of widths for embeddings of
Besov classes (see, e.g., \cite{vybiral}). Let us formulate the
final result.

Let $r\in \N$, $1\le p, \, q\le \infty$. Denote $\eta _{pq}=\frac
12 \cdot \frac{\frac 1p -\frac 1q}{\frac 12 -\frac 1q}$,
$\varkappa =\left(\frac rd +\frac 1q-\frac 1p\right)^{-1}$.
\begin{trma}
Denote
$$
\theta _{p,q,r,d}=\left\{ \begin{array}{l} \frac rd , \ \ \text{if
}\ p\ge q \, \text{ or }\,
(2\le p<q\le \infty , \; \frac rd \ge \eta _{pq}), \\
\frac rd +\frac 1q-\frac 1p,\ \ \text{ if }\ 1\le p<q\le 2, \\
\frac rd + \frac 12- \frac 1p, \ \ \text{ if }\ 1<p<2<q\le \infty
\,\text{ and }\,\frac rd \ge\frac 1p, \\
\frac q2 \left(\frac rd +\frac 1q- \frac 1p\right),\ \text{ if }\
(p<2<q , \; \frac rd <\frac 1p)\, \text{ or }\, (2\le p<q, \ \frac
rd <\eta _{pq}),
\end{array}\right.
$$
$\tilde \theta_{p,q,r,d}=\theta_{p,q,r,d}$ for $\frac 1p+\frac
1q\ge 1$, $\tilde \theta_{p,q,r,d}=\theta_{q',p',r,d}$ for $\frac
1p+\frac 1q<1$. Let $\displaystyle \frac rd +\frac 1q-\frac 1p>0$.
Suppose that $\frac rd\ne \frac 1p$ holds in the case $1\le
p<2<q\le +\infty$, and $\frac rd \ne \eta_{pq}$ holds in the case
$2\le p<q\le +\infty$. Then
$$
d_n(W^r_p[0, \, 1]^d, \, L_q[0, \,
1]^d)\underset{r,d,p,q}{\asymp}n^{-\theta_{p,q,r,d}}.
$$
Suppose that $\frac rd\ne \max\left\{\frac{1}{p}, \,
\frac{1}{q'}\right\}$ holds in the case $1\le p<2<q\le +\infty$. Then
$$
\lambda_n(W^r_p[0, \, 1]^d, \, L_q[0, \,
1]^d)\underset{r,d,p,q}{\asymp}n^{-\tilde\theta_{p,q,r,d}}.
$$
\end{trma}

Let us formulate the main result of this paper. Denote by
$L_+(\Omega)$ the class of functions $w:\Omega \rightarrow \R_+$
such that there exists a sequence of functions $w_n:\Omega
\rightarrow \R_+$, $n\in \N$, with the following properties:
\begin{itemize}
\item $0 \le w_n(x) \le w(x)$ for any $x \in \Omega$;
\item there exists a finite family of non-overlapping cubes $K_{n,i}\subset
\Omega$, $1\le i\le N_n$, such that $w_n|_{K_{n,i}}= {\rm const}$,
$w_n(x)=0$ for $x\in \Omega\backslash \cup _{i=1}^{N_n}K_{n,i}$;
\item $w_n(x)\underset{n\rightarrow \infty}{\rightarrow} w(x)$
a.e. on $\Omega$.
\end{itemize}
For sets $A$, $B\subset \R^d$ and for a point $x\in \R^d$ we write
$|x|=\|x\|_{l_2^d}$, $\dist \, (x, \, A)=\inf _{y\in A}|x-y|$,
$\dist \, (A, \, B)=\inf _{y\in A, \, z\in B}|y-z|$.

\begin{Trm}
\label{main} Let $\Omega \subset \R^d$ be a bounded domain such
that $\Omega \in {\bf FC}(a)$, let $r\in \N$, $1<p\le \infty$,
$1\le q<\infty$, and let $\frac rd +\frac 1q-\frac 1p >0$. Let
$\Gamma'$, $\Gamma''\subset \partial \Omega$ be closed sets, let
$g(x)=g_0(x)\tilde g(x)$, $v(x)=v_0(x)\tilde v(x)$, $x\in \Omega$,
$g_0\in L_\alpha(\Omega , \, \R_+)$, $v_0\in L_\beta(\Omega , \,
\R_+)$, $1<\alpha, \, \beta \le \infty$, $\beta>q$, $\frac
1p+\frac{1}{\alpha}<1$, $\frac{1}{\tilde \varkappa}:=\frac
rd+\frac 1q-\frac{1}{\beta}-\frac 1p-\frac{1}{\alpha}\ge 0$;
suppose that if $\frac{1}{\tilde \varkappa}=0$, then $\tilde
g=\tilde v=1$, and if $\frac{1}{\tilde \varkappa}>0$, then $\tilde
g\tilde v\in L_{\tilde \varkappa}(\Omega)$,
\begin{align}
\label{tg_dg} \tilde g(x)=\varphi_{\tilde g}\left({\rm dist}\, (x,
\, \Gamma')\right), \;\; \tilde v(x) =\varphi_{\tilde v}\left({\rm
dist}\, (x, \, \Gamma'')\right);
\end{align}
here the function $\varphi_{\tilde g}:(0, \, +\infty)\rightarrow
\R_+$ decreases, the function $\varphi_{\tilde v}:(0, \,
+\infty)\rightarrow \R_+$ increases, and there exists a number
$c_0\ge 1$ such that for any $m\in \Z$, $t$, $s\in [2^{m-1}, \,
2^{m+1}]$
\begin{align}
\label{delta2}
c_0^{-1}\le \frac{\varphi_{\tilde g}(t)}{\varphi_{\tilde g}(s)}\le c_0, \;\;\;
c_0^{-1}\le \frac{\varphi_{\tilde v}(t)}{\varphi_{\tilde v}(s)}\le c_0.
\end{align}
\begin{enumerate}
\item Suppose that $\frac rd\ne
\frac 1p$ holds for $1<p<2<q<+\infty$, and $\frac rd \ne
\eta_{pq}$ holds for $2\le p<q<+\infty$. Then
$$
\varlimsup _{n\rightarrow \infty} n^{\theta_{p,q,r,d}}
d_n(W^r_{p,g}(\Omega), \, L_{q,v}(\Omega))\underset{r,d,q,p,a,
\alpha , \beta,c_0}{\lesssim} \| gv\| _{\varkappa}.
$$
\item Suppose that $\frac rd\ne \max\left\{\frac{1}{p},
\, \frac{1}{q'}\right\}$ holds for $1<p<2<q<+\infty$. Then
$$
\varlimsup _{n\rightarrow \infty} n^{\tilde\theta_{p,q,r,d}} {\cal
A}_n(W^r_{p,g}(\Omega), \, L_{q,v}
(\Omega))\stackrel{(\ref{aneqln})}{=}
$$
$$
=\varlimsup _{n\rightarrow
\infty} n^{\tilde\theta_{p,q,r,d}}
\lambda_n(W^r_{p,g}(\Omega), \,L_{q,v} (\Omega))
\underset{r,d,q,p,a,\alpha ,\beta,c_0}{\lesssim}
\| gv\| _{\varkappa}.
$$
\end{enumerate}
If $g$, $v\in L_+(\Omega)$, then
$$
\varliminf _{n\rightarrow
\infty} n^{\theta_{p,q,r,d}} d_n(W^r_{p,g}(\Omega), \, L_{q,v}
(\Omega))\underset{r,d,q,p}{\gtrsim} \| gv\| _{\varkappa},
$$
$$
\varliminf _{n\rightarrow \infty} n^{\tilde\theta_{p,q,r,d}} {\cal
A}_n(W^r_{p,g}(\Omega), \, L_{q,v}
(\Omega))\stackrel{(\ref{aneqln})}{=}\varliminf _{n\rightarrow
\infty} n^{\tilde\theta_{p,q,r,d}} \lambda_n(W^r_{p,g}(\Omega), \,
L_{q,v} (\Omega))\underset{r,d,q,p}{\gtrsim} \| gv\| _{\varkappa}.
$$
\end{Trm}

For common domains and $r=1$, $p=q$ Evans, Edmunds and Harris
\cite{edm_ev_06, har_ev93} obtained a sufficient condition
under which the approximation numbers have the same orders
as for a cube. In addition, note the results of Evans, Harris, Lang and Solomyak
\cite{e_h_l, solomyak} on approximation numbers of weighted Sobolev classes on
a metric graph for $r=1$, $p=q$. Also the author knows the recent Besov's result
on coincidence of orders of widths
$$
d_n(W^r_p(K_\sigma), \, L_q(K_\sigma))\underset{p,q,r,d,\sigma}
{\asymp} d_n(W^r_p([0, \, 1]^d), \, L_q([0, \, 1]^d)),
$$
where $$K_\sigma=\{(x_1, \, \dots, \, x_{d-1}, \, x_d): \; |(x_1, \,
\dots, \, x_{d-1})|^{1/\sigma}<x_d<1\},$$ $\sigma>1$,
$r-[\sigma(d-1)+1]\left(\frac 1p-\frac 1q\right)_+>0$.

\section{Notations}
We denote by $\overline{A}$, or ${\rm int}\, A$, or ${\rm
mes}\,A$, or ${\rm card}\,A$ the closure of the set $A$, or its
interior, or its Lebesgue measure or its cardinality,
respectively. If the set $A$ is contained in some subspace
$L\subset \R^d$ of dimension $(d-1)$, then we denote by ${\rm
int}_{d-1}A$ the interior of the set $A$ with respect to the
induced topology of the space $L$. We say that the sets $A$,
$B\subset \R ^d$ do not overlap if $A\cap B$ has the Lebesgue
measure zero. For a convex set $A$ we denote by $\dim A$ the
dimension of the affine span of the set $A$.

Let $\gamma$ be a rectifiable curve in $\R^d$. We shall denote by $|\gamma |$
its length.

Let ${\cal K}$ be a family of closed cubes in $\R^d$ with axes
parallel to coordinate axes. For a cube $K \in {\cal K}$ and for
$s\in \Z_+$ we denote by $\Xi _s(K)$ the set of $2^{sd}$ closed
non-overlapping cubes of the same size that form a partition of $K$,
and write $\Xi(K):=\bigcup_{s\in \Z_+} \Xi _s(K)$. We note the
following property of $\Xi(K)$, $K\in {\cal K}$: if $\Delta _1$,
$\Delta _2\in \Xi(K)$, then either $\Delta _1$ and $\Delta _2$ do
not overlap or we have either $\Delta _1\in \Xi(\Delta _2)$ or
$\Delta _2\in \Xi(\Delta _1)$.

Denote by $\chi_E$ an indicator function of a set $E$.

We recall some definitions from graph theory. Throughout, we
assume that the graphs have neither multiple edges nor loops.

Let $\Gamma$ be a graph which contains no more than a countable
number of vertices. We shall denote by ${\bf V}(\Gamma)$ and by
${\bf E}(\Gamma)$ the set of vertices and the set of edges of
$\Gamma$, respectively. Two vertices are called adjacent if there
is an edge between them. We shall identify pairs of adjacent
vertices with edges connecting them. If a vertex is an endpoint of
an edge, we say that these vertex and edge are incident. If
$v_i\in {\bf V}(\Gamma)$, $1\le i\le n$, the vertices $v_i$ and
$v_{i+1}$ are adjacent for any $i=1, \, \dots , \, n-1$, then the
sequence $(v_1, \, \dots, \, v_n)$ is called a path of length
$n-1$. If all vertices $v_i$ are distinct, then such a path is
called simple. If $n\ge 4$, $(v_1, \, \dots, \, v_{n-1})$ is a
simple path and $v_1=v_n$, then such a path is called a cycle. We
say that a path $(v_1, \, \dots, \, v_{n-1}, \, v_n)$ is almost
simple, if the path $(v_1, \, \dots, \, v_{n-1})$ is simple (in
particular, simple paths and cycles are almost simple). Let
$\Gamma$ be a directed graph, let $v_i$ be a head of the arc
$(v_i, \, v_{i+1})$, and let $v_{i+1}$ be its tail for any $i=1,
\, \dots , \, n-1$. Then we say that the path $(v_1, \, \dots, \,
v_n)$ is directed; here $v_1$ is the origin and $v_n$ is the
destination of this path. We say that a graph is connected if
there is a finite path from any vertex to any other vertex in the
graph. If a connected graph has no cycles, then it is called a
tree.

Let $({\cal T}, \, v_0)$ be a tree with a distinguished vertex
(or a root) $v_0$. Then a partial order on ${\bf V}({\cal T})$
is introduced as follows: we say that $v'>v$ if there exists a path
$(v_0, \, v_1, \, \dots , \, v_n, \, v')$ such that $v=v_k$ for some
$k\in \overline{0, \, n}$. In this case we put $\rho(v, \, v')=\rho(v', \, v)=n+1-k$
and call this value the distance between $v$ and $v'$.
In addition, put $\rho(v, \, v)=0$. If $v'>v$ or $v'=v$, then we write $v'\ge v$
and put $[v, \, v']:= \{v''\in {\bf V}({\cal T}):v\le v''\le v'\}$. Denote by
\label{v1v}${\bf V}_1(v)$ the set of vertices $v'>v$ such that $\rho(v, \, v')=1$.
Let $v\in {\bf V}({\cal T})$. Denote by ${\cal T}_v=({\cal T}_v, \, v)$
a subtree of ${\cal T}$ with a set of vertices
\begin{align}
\label{vpvtvpv} \{v'\in {\bf V}({\cal T}):v'\ge v\}.
\end{align}
The introduced partial order on ${\cal T}$ induces a partial order on
its subtree.

We notice the following property of a tree $({\cal T}, \, v_0)$: if its
vertices $v'$ and $v''$ are incomparable, then ${\cal T}_{v'}\cap {\cal
T}_{v''}=\varnothing$.

Let ${\cal T}$ be a tree. Denote by ${\bf ST}({\cal T})$
the set of subtrees in ${\cal T}$. If ${\cal T}_1$,
${\cal T}_2\in {\bf ST}({\cal T})$ and ${\bf V}({\cal T}_1)
\subset {\bf V}({\cal T}_2)$, then we say that ${\cal T}_1
\subset {\cal T}_2$.

Let $W\subset {\bf V}({\cal T})$. We say that $W\in {\bf
VST}({\cal T})$ if $W={\bf V}({\cal T}')$ for some ${\cal
T}'\in {\bf ST}({\cal T})$. Notice that ${\bf V}({\cal
T}_1)\cap {\bf V}({\cal T}_2)\in {\bf VST}({\cal T})$ holds for any trees ${\cal
T}_1$, ${\cal T}_2\in {\bf ST}({\cal T})$.
Denote by ${\cal T}_1\cap {\cal T}_2$ a subtree with a set of vertices
${\bf V}({\cal T}_1)\cap {\bf V}({\cal T}_2)$. If ${\bf
V}({\cal T}_1)\cup {\bf V}({\cal T}_2)\in {\bf VST}({\cal T})$
(or ${\bf V}({\cal T}_1)\backslash {\bf V}({\cal
T}_2)\in {\bf VST}({\cal T})$), then  we denote by
${\cal T}_1\cup {\cal T}_2$ (${\cal T}_1\backslash
{\cal T}_2$, respectively) a subtree with a set of vertices
${\bf V}({\cal T}_1)\cup {\bf V}({\cal T}_2)$ (or
${\bf V}({\cal T}_1)\backslash {\bf V}({\cal T}_2)$, respectively).
If ${\bf V}({\cal T}_1)\cap {\bf V}({\cal
T}_2)=\varnothing$, then we write ${\cal T}_1\cup
{\cal T}_2={\cal T}_1\sqcup {\cal T}_2$.

Let ${\cal T}$, ${\cal T}_1, \, \dots , \, {\cal T}_k$ be trees that have no
common vertices, and let $v_1, \, \dots , \, v_k\in {\bf
V}({\cal T})$, $w_j\in {\bf V}({\cal T}_j)$, $j=1, \, \dots , \,
k$ ($k\in \N\cup \{\infty\}$). Denote by
$$
{\bf J}({\cal T}, \, {\cal T}_1, \, \dots , \, {\cal T}_k;
v_1, \, w_1, \, \dots , \, v_k, \, w_k)
$$
a tree obtained from ${\cal T}$, ${\cal T}_1, \, \dots , \, {\cal
T}_k$ by an edge connecting the vertex $v_j$ with the vertex $w_j$
for $j=1, \, \dots , \, k$.

\section{Auxiliary assertions}
Let $\Theta\subset \Xi([0, \, 1]^d)$ be a set of non-overlapping
cubes.

\begin{Def}
Let ${\cal G}$ be a graph, and let $F:{\bf V}({\cal G})
\rightarrow \Theta$ be a one-to-one mapping. We say that
$F$ is consistent with the structure of the graph ${\cal G}$ if
the following condition holds: for any adjacent vertices $v'$,
$v''\in {\bf V}({\cal G})$ the set $\Gamma _{v',v''} :=F(v')\cap
F(v'')$ has dimension $d-1$.
\end{Def}

{\bf Remark.} If the mapping $F$ is consistent with the structure
of a graph ${\cal G}$, the vertices $v'$ and $v''$ are adjacent and ${\rm mes}\,
F(v') \ge {\rm mes}\, F(v'')$, then $F(v')\cap F(v'')$ is a
$(d-1)$-dimensional face of the cube $F(v'')$.

Let $({\cal T}, \, v_*)$ be a tree, and let $F:{\bf V}({\cal T})
\rightarrow \Theta$ be a one-to-one mapping consistent with the structure
of the tree ${\cal T}$. For any adjacent vertices $v'$ and $v''$ we set
$\mathaccent '27 \Gamma _{v',v''}=\inter _{d-1}\Gamma _{v',v''}$, and for
each subtree ${\cal T}'$ of ${\cal T}$ we put
\begin{align}
\label{def_dom_by_tree}
\Omega _{{\cal T'},F}=\left(\cup _{v\in {\bf V}({\cal T'})}\inter F(v)\right)
\cup \left(\cup _{(v',v'')\in {\bf E}({\cal T'})}\mathaccent '27 \Gamma _{v',v''}\right).
\end{align}
If $v\in {\bf V}({\cal T})$ and $\Delta=F(v)$, then we denote $\Omega_{\le
\Delta}=\Omega_{[v_*, \, v],F}$.

Let $v'$, $v''$ be adjacent vertices of ${\cal T}$, let $\Gamma
_{v',v''}$ coincide with a $(d-1)$-dimensional face of $F(v')$
(then ${\rm mes}\, F(v') \le {\rm mes}\, F(v'')$), and let
$x'$, $x''$ be centers of the cubes $F(v')$ and $F(v'')$,
respectively. Denote by $y$ the orthogonal projection of the point
$x'$ onto $\Gamma _{v',v''}$, and set
$$
\gamma _{v'v''}(t)=\left\{ \begin{array}{l}
\frac{|x'-y|-t}{|x'-y|}x'+\frac{t}{|x'-y|}y, \;\; 0\le t\le |x'-y|, \\
\frac{|x''-y|+|x'-y|-t}{|x''-y|}y+\frac{t-|x'-y|}{|x''-y|}x'', \;\;
|x'-y|\le t\le |x''-y|+|x'-y|,
\end{array}\right .
$$
$$
\gamma _{v''v'}(t)=\left\{ \begin{array}{l}
\frac{|x''-y|-t}{|x''-y|}x''+\frac{t}{|x''-y|}y, \;\; 0\le t\le |x''-y|, \\
\frac{|x'-y|+|x''-y|-t}{|x'-y|}y+\frac{t-|x''-y|}{|x'-y|}x', \;\;
|x''-y|\le t\le |x''-y|+|x'-y|.
\end{array}\right .
$$

Let $\Delta \in \Xi([0, \, 1]^d)$. Denote by ${\bf m}(\Delta)$ such
$m\in \N$ that $\Delta \in \Xi _m([0, \, 1]^d)$. For any vertex $v\in {\bf V}({\cal T})$
put $m_v={\bf m}(F(v))$.
\begin{Lem}
\label{omega_t_cone} Let $({\cal T}, \, v_*)$ be a tree,
and let $F:{\bf V}({\cal T}) \rightarrow \Theta$ be consistent with the structure
of ${\cal T}$. Suppose that there exist $l_*$,
$k_*\in \N$ such that for any vertices $v', \, v''\in {\bf V}({\cal
T})$, $v'>v''$
\begin{align}
\label{mvsmvkvskvms}
l_*(m_{v'}-m_{v''})\ge \rho(v', \, v'')-k_*.
\end{align}
Then there exists $\hat a=\hat a(k_*, \, l_*, \, d)$ such that for
any subtree ${\cal T}'$ of ${\cal T}$ the set $\Omega
_{{\cal T}',F}$ is a domain belonging to the class ${\bf FC}(\hat a)$.
Here the curve $\gamma_x$ from Definition \ref{fca} can be chosen so that
\begin{align}
\label{incl} B_{\hat at}(\gamma_x(t))\subset \Omega_{\le F(w)},
\text{ if }x\in F(w).
\end{align}
In addition,
\begin{align}
\label{vol} {\rm mes}\, \Omega_{{\cal T}',F}\underset{\hat
a,d}{\asymp}{\rm mes}\, F(v),
\end{align}
where $v$ is the minimal vertex of ${\cal T}'$.
\end{Lem}
\begin{proof}
Let ${\cal T}'$ be a subtree in ${\cal T}$, let $v$ be the minimal
vertex of ${\cal T}'$, let $x_v$ be the center of the cube $F(v)$,
and let $z\in \Omega _{{\cal T}',F}$.
Then $z\in F(v')$, where $v'$ is a vertex of ${\cal T}'$. Define
the curve $\gamma _z$ with the starting point $z$ and the endpoint $x_v$
as follows. Let $v_1>v_2> \dots >v_k$ be a sequence of vertices in
${\cal T}'$ such that $v_1=v'$,
$v_k=v$, $\rho(v_j, \, v_{j+1})=1$, $j=1, \, \dots , \, k-1$. Denote by $x_j$
the center of $F(v_j)$, $s_j=|\gamma _{v_{j-1}v_j}|$, $\tau _1=|z-x_1|$,
$\tau _j=\tau _{j-1}+s_j$, $2\le j\le k$,
$$
\gamma(t)=\left \{ \begin{array}{l} \frac{\tau _1-t}{\tau _1}z
+\frac{t}{\tau _1}x_1, \;\; 0\le t\le \tau _1, \\ \gamma
_{v_{j-1}v_j}(t-\tau_{j-1}), \;\; \tau _{j-1}\le t\le \tau _j,
\;\; 2\le j\le k,\end{array}\right .
$$
$$
E_t=\left \{ \begin{array}{l} F(v_1), \;\; 0\le t\le \tau _1, \\
F(v_{j-1})\cup F(v_j), \;\; \tau _{j-1}\le t\le \tau _j, \;\; 2\le
j\le k.\end{array}\right .
$$
For any $t\in [0, \, \tau _k]$ denote by $a_t$
the maximal radius of an open ball centered at $\gamma(t)$
that is contained in $E_t$. Show that
\begin{align}
\label{atgtrsimt} a_t\underset{d,k_*,l_*}{\gtrsim}t.
\end{align}
This will imply the first assertion of Lemma and (\ref{incl}).

Denote by $\sigma _j$ the length of the side of $F(v_j)$.
Then $\tau _1\underset{d}{\lesssim}\sigma _1$, $s_j\underset{d}{\asymp}
\max\{ \sigma _j, \, \sigma _{j-1}\}$. Notice that for $t\in [0, \, \tau _1]$
the inequality $a_t\underset{d}{\gtrsim}t$ holds. Let $j\ge 2$, $\tau _{j-1}\le t\le \tau _j$,
and let $\tilde t_j\in [\tau _{j-1}, \, \tau _j]$ be such that $\gamma(\tilde t_j)
\in \Gamma _{v_{j-1},v_j}$.
We have
\begin{align}
\label{tau_j1}
\displaystyle
\begin{array}{c}
\tau _{j-1}=\tau _1+\sum \limits _{i=2}^{j-1}
s_i\underset{d}{\lesssim} \sigma _1+\sum \limits _{i=2}^{j-1}\max
\{\sigma _i, \, \sigma _{i-1}\} \le \sum \limits _{i=1}^{j-1}
2\sigma _i =\sum \limits _{i=1}^{j-1} 2^{-m_{v_i}+1}=\\
=2^{-m_{v_{j-1}}+1}\sum \limits _{i=1}^{j-1}
2^{-m_{v_i}+m_{v_{j-1}}} \stackrel{(\ref{mvsmvkvskvms})}
{\underset{d,k_*,l_*}{\lesssim}}\sigma _{j-1} \sum \limits
_{i=1}^{j-1} 2^{\frac{i-j}{l_*}} \underset{l_*}{\lesssim} \sigma
_{j-1}.
\end{array}
\end{align}
Show that
\begin{align}
\label{atundkldttj}
a_t\underset{k_*,l_*,d}{\gtrsim}
\max\{t-\tilde t_j, \, \sigma _{j-1}\}.
\end{align}
Indeed, if $t\in [\tau_{j-1}, \, \tilde t_j)$, then
$\max\{t-\tilde t_j, \, \sigma _{j-1}\}=\sigma_{j-1}$; if $t\in
[\tilde t_j, \, \tau _j]$, then $a_t\ge \frac{t-\tilde t_j}{2(\tau
_j-\tilde t_j)}\sigma _j= \frac{t-\tilde t_j}{2|\gamma(\tilde
t_j)-x_j|}\sigma _j \underset{d}{\gtrsim}t-\tilde t_j$. From
(\ref{mvsmvkvskvms}) follows that $\min \{\sigma _j, \, \sigma
_{j-1}\}\underset{k_*,l_*,d}{\gtrsim} \sigma _{j-1}$. Finally,
$a_t\underset{d}{\gtrsim}\min \{\sigma _j, \, \sigma _{j-1}\}$,
which yields (\ref{atundkldttj}).

Let us prove (\ref{atgtrsimt}). If $t-\tilde t_j\le \sigma _{j-1}$, then
$$
t\le \tilde t_j+\sigma _{j-1}= \tau _{j-1}+(\tilde t_j-\tau
_{j-1})+\sigma
_{j-1}\stackrel{(\ref{tau_j1})}{\underset{d,l_*,k_*}{\lesssim}}
\sigma _{j-1}\stackrel{(\ref{atundkldttj})}
{\underset{d,l_*,k_*}{\lesssim}}a_t.
$$
Let $t-\tilde t_j>\sigma _{j-1}$. If $t-\tilde t_j\le \frac{t}{2}$, then
$$
t\le 2\tilde t_j=2\tau _{j-1}+2(\tilde t_j-\tau _{j-1})
\stackrel{(\ref{tau_j1})}{\underset{d,l_*,k_*}{\lesssim}}\sigma
_{j-1} \stackrel{(\ref{atundkldttj})}
{\underset{d,l_*,k_*}{\lesssim}}a_t.
$$
If $t-\tilde t_j>\frac{t}{2}$, then
$a_t\stackrel{(\ref{atundkldttj})}{\underset{k_*,l_*,d}{\gtrsim}}t-\tilde
t_j>\frac{t}{2}$.

The relation (\ref{vol}) follows from the inclusion $F(v)\subset
\Omega_{{\cal T}',F}$ and from the estimate ${\rm diam}\,\Omega_{{\cal
T}',F}\underset{\hat a,d}{\asymp} {\rm diam}\, F(v)$ (which is the consequence
of the definition of ${\bf FC}(\hat a)$).
\end{proof}
\begin{Def}
\label{dgk} Let $\Gamma$ be a finite direct graph, let
$v_*\in {\bf V}(\Gamma)$ and $k\in \N$. We say that $(\Gamma , \,
v_*)\in \mathfrak{G}_k$ if for any $v\in {\bf V}(\Gamma)$
there exists a simple directed path with the origin $v_*$ and
the destination $v$ such that the length of this path does not exceed $k$.
\end{Def}
\begin{Lem}
\label{graph_tree} Let $\Gamma$ be a finite directed
graph, let $v_*\in {\bf V}(\Gamma)$, $k\in \N$, and let $(\Gamma , \, v_*)\in
\mathfrak{G}_k$. In addition, suppose that $v_*$ is the head of
all edges incident to $v_*$. Then there exists a tree
${\cal T}$ rooted at $v_*$, which is a subgraph of $\Gamma$ such that
${\bf V}({\cal T})= {\bf V}(\Gamma)$ and for any $v\in {\bf
V}(\cal T)$ the inequality $\rho(v, \, v_*)\le k$ holds.
\end{Lem}
\begin{proof}
Notice that the graph $\Gamma$ is connected. Denote by ${\cal S}$
the set of directed almost simple paths with the origin $v_*$
(this set is finite). Let $v\ne v_*$. Denote by ${\cal S}_v$ the
set of paths belonging to ${\cal S}$ with the destination $v$. We
say that the path belongs to $\tilde {\cal S}=\tilde {\cal
S}(\Gamma)$ if it belongs to ${\cal S}_v$ for some $v\ne v_*$ and
${\rm card}\, {\cal S}_v\ge 2$.

Show that
\begin{enumerate}
\item if $\tilde {\cal S}=\varnothing$, then $\Gamma$ is a tree;
\item if $\tilde {\cal S}\ne \varnothing$, then there exists a graph $\tilde
\Gamma$, which is obtained from $\Gamma$ by removing an edge
(while preserving the vertices), such that $(\tilde\Gamma , \,
v_*)\in \mathfrak{G}_k$.
\end{enumerate}
Employing the preceding statements, by induction we attain the
assertion of the Lemma. Indeed, set $\Gamma _0=\Gamma$. Let for
some $n\in \Z_+$ the graph $\Gamma _n$ be constructed by removing
$n$ edges in $\Gamma$ such that $(\Gamma _n, \, v_*)\in
\mathfrak{G}_k$. If $\tilde {\cal S}(\Gamma_n)=\varnothing$, then
$\Gamma_n$ is the desired tree, otherwise we apply assertion 2 to
$(\Gamma_n, \, v_*)$ and define the graph $\Gamma_{n+1}$. Since
the graph $\Gamma$ is finite, then there exists $m\in \N$ such
that $\tilde{\cal S}(\Gamma_m)=\varnothing$.

Let us prove assertion 1. Assume that $\Gamma$ is not a tree. Then
$\Gamma$ has a cycle. Hence, there exists a vertex $v$, which can
be connected with $v_*$ via two distinct simple paths (not
necessarily directed). Since $(\Gamma , \, v_*)\in
\mathfrak{G}_k$, then there exists a simple directed path $s=
(v_*, \, v_1, \, \dots , \, v_n)$ with the origin $v_*$ and the
destination $v=v_n$. Let
$$
\overline{s}=(v_*, \, v_1', \, \dots , \, v_l', \, v_{m+1}, \,
\dots , \, v_n)
$$
be another simple path connecting $v_*$ and $v$, and let $v_l'\ne v_m$.
Put $v'_{l+1}=v_{m+1}$, $v'_{l+2}=v_m$, $v'_0=v_*$. Let
${\bf W}\subset \{v'_0, \, \dots, \, v'_l, \, v'_{l+1}\}$ be
the set of vertices $v_j'$ that are heads of $(v'_j, \,
v'_{j+1})$. By the condition of Lemma, $v'_0\in {\bf W}$. In addition,
$v'_{l+1}\notin {\bf W}$ (since the path $s$ is directed and its origin is
$v_*$). Let
$$
\hat l=\max\{j\in \overline{0, \, l}:\; v'_j\in {\bf W}\}.
$$
Show that ${\rm card}\, {\cal S}_{v_{\hat{l}+1}'} \ge 2$. Indeed,
by the condition of Lemma, there exist directed simple paths
$\sigma '\in {\cal S}_{v'_{\hat l}}$ and $\sigma ''\in {\cal
S}_{v'_{\hat l+2}}$. Appending the vertex $v'_{\hat{l}+1}$, we
obtain two paths which belong to ${\cal S}_{v'_{\hat l+1}}$
(notice that $v'_{\hat{l}+1}\notin \{v'_{\hat{l}}, \, v'_{{\hat
l}+2}\}$). In order to prove that these paths are different, it is
sufficient to check that $v'_{\hat{l}}\ne v'_{{\hat l}+2}$.
Indeed, if $\hat{l}\le l-1$, then it is true, since the path
$\overline{s}$ is simple; if $\hat{l}=l$, then it follows from the
condition $v'_l\ne v_m$.

Prove assertion 2. Actually, there exists the path $(v_*, \, v_1,
\, \dots , \, v_{k_0})\in \tilde {\cal S}$ such that
\begin{align}
\label{k0maxssssss}
k_0=\max \{|s|:s\in \tilde {\cal S}\},
\end{align}
where $|s|$ is the length of the path $s$. Since ${\rm card}\,
{\cal S}_{v_{k_0}}\ge 2$, then by (\ref{k0maxssssss}) and by the
condition $\Gamma \in \mathfrak{G}_k$, there exists a direct path
$s'=(v_*, \, v_1', \, \dots , \, v_l', \, v_{k_1}, \, \dots , \,
v_{k_0}) \in {\cal S}_{v_{k_0}}$ such that
\begin{align}
\label{sssminstils}
|s'|=\min \{|s|:s\in {\cal S}_{v_{k_0}}\}
\end{align}
and $v_l'\ne v_{k_1-1}$. Then the path $s'$ is simple. Remove the
edge $(v_{k_1-1}, \, v_{k_1})$ and obtain the graph $\tilde
\Gamma$. Show that $(\tilde\Gamma , \, v_*)\in \mathfrak{G}_k$.
Let $v\in {\bf V}(\tilde \Gamma)$. Then there exists a direct
simple path $s_v\in {\cal S}_v$ in $\Gamma$ such that $|s_v|\le
k$. If $s_v$ does not contain the edge $(v_{k_1-1}, \, v_{k_1})$,
then it is a path in $\tilde \Gamma$. Let $s_v$ contain the edge
$(v_{k_1-1}, \, v_{k_1})$, that is, $s_v$ is composed of two paths
$$
s_1=(v_*,  \, w_1, \, \dots , \, w_\nu , \, v_{k_1-1},
\, v_{k_1}) \;\;\text{ and }\;\; s_2=(v_{k_1}, \, u_1, \, \dots , \, u_\mu , \, v).
$$
Consider the path
$$
s''=(v_*, \, v_1', \, \dots , \, v_l', \, v_{k_1}, \, u_1, \, \dots , \, u_\mu , \, v).
$$
Then $s''$ does not contain the edge $(v_{k_1-1}, \, v_{k_1})$. Removing if necessary
cyclic sections in $s''$, we obtain a simple path
$s'''\in {\cal S}_v$ without the edge $(v_{k_1-1}, \,
v_{k_1})$. It remains to prove that $|s'''|\le k$. Indeed,
$$
|s'''|\le |s''|\le l+1+|s_2|
\stackrel{(\ref{sssminstils})}{\le}|s_1|+|s_2|=|s_v|\le k.
$$
This completes the proof.
\end{proof}

Formulate the Whitney covering theorem (see, e.g.,
\cite{leoni}, page 562).
\begin{trma}
\label{whitney} Let $\Omega\subset (0, \, 1)^d$ be an open set. Then
there exists a family of closed pairwise non-overlapping cubes $\Theta(\Omega)=
\{\Delta_j\}_{j\in\N}\subset \Xi([0, \, 1]^d)$ with the following
properties:
\begin{enumerate}
\item $\Omega=\cup _{j\in \N}\Delta_j$;
\item ${\rm dist}\, (\Delta_j, \, \partial \Omega)\underset{d}{\asymp} 2^{-{\bf m}(\Delta_j)}$;
\item for any $j\in \N$
\begin{align}
\label{card_sopr}
{\rm card}\, \{i\in \N:\dim (\Delta_i\cap \Delta_j)=d-1\}\le 12^d.
\end{align}
\end{enumerate}
\end{trma}

\begin{Lem}
\label{constr_omega_t} Let $a>0$, $\Omega \in {\bf FC}(a)$,
$\Omega\subset (0, \, 1)^d$. Then there exists a tree $({\cal T},
\, v_0)$, a mapping $F:{\bf V}({\cal T}) \rightarrow
\Theta(\Omega)$ consistent with the structure of ${\cal T}$,
and numbers $k_*=k_*(a, \, d)\in \N$, $l_*=l_*(a, \, d)\in \N$ such that
for any vertices $v'>v''$ the inequality (\ref{mvsmvkvskvms}) holds.
\end{Lem}
\begin{proof}
Index the cubes $\{\Delta _{i,j}\}_{j\in \Z_+, \, 1\le i\le r_j}$
from the set $\Theta(\Omega)$ so that $r_0=1$, ${\bf m}(\Delta
_{i,j})=\mu _j\in \Z_+$, $1\le i\le r_j$, and
\begin{align}
\label{mu0mu1}
\mu _0\le \mu _1, \;\; \mu _j<\mu _{j+1},\;\; j\in
\N.
\end{align}

Let $x_*$ be the center of $\Delta _{1,0}$. Since $\Omega \in {\bf
FC}(a)$, then there exists $b=b(a, \, d)$ such that for any $x\in
\Omega$ there exists a polygonal arc $\gamma _x:[0, \, T_0(x)]
\rightarrow\Omega$ with the following properties:

a) $\gamma _x\in AC[0, \, T_0(x)]$, $|\dot \gamma _x|=1$ a.e.;

b) $\gamma _x(0)=x$, $\gamma _x(T_0(x))=x_*$;

c) for any $t\in [0, \, T_0(x)]$ the inclusion
$B_{bt}(\gamma _x(t))\subset \Omega$ holds;

d) arcs of $\gamma _x$ are not parallel to any
$d-1$-dimensional coordinate planes and for each $\Delta \in
\Theta(\Omega)$ the set $\gamma _x([0, \, T_0(x)])$ does not intersect
$\Delta$ at any point of $k$-dimensional faces, $k\le
d-2$.

Construct by induction sequences of trees $\{{\cal
T}_n\}_{n\in \Z_+}$ rooted at $v_0$, families of cubes
$\Theta_n\subset \Theta(\Omega)$ and one-to-one mappings
$F_n:{\bf V}({\cal T}_n)\rightarrow \Theta_n$ consistent with
the structure of ${\cal T}_n$ and satisfying the following conditions:
\begin{enumerate}
\item ${\cal T}_n$ is a subtree of ${\cal T}_{n+1}$,
$F_{n+1}|_{{\bf V}({\cal T}_n)}=F_n$, $\Theta_n\supset \{\Delta
_{i,n}\}_{1\le i\le r_n}$, $n\in \N$, $F_0(v_0)=\Delta _{1,0}$;
\item there exists $c_*(d, \, a)\in \N$ such that for any vertex $v$
of the tree ${\cal T}_n$ the inequality $m_v\le \mu _n+c_*(d, \, a)$ holds;
\item there exists $c_{**}(d, \, a)\in \N$ such that if
$v , \, v'\in{\cal T}_n\backslash {\cal T}_{n-1}$, $v'>v$,
then $\rho(v', \, v)\le c_{**}(d, \, a)$;
\item ${\cal T}_n$ satisfies the condition of Lemma \ref{omega_t_cone} with
\begin{align}
\label{l_st} l_*=1+(c_{**}(d, \, a)+1)(c_*(d, \, a)+2),
\end{align}
\begin{align}
\label{k_st} k_*=l_*c_*(d, \, a)+(c_{**}(d, \, a)+1) (c_*(d, \,
a)+2).
\end{align}
\end{enumerate}

As ${\cal T}_0$ we take a tree consisting of the unique vertex
$v_0$; put $\Theta_0=\{\Delta _{1,0}\}$, $F_0(v_0)=\Delta _{1,0}$. The property 1
holds by construction, the property 2 follows from the equality
$m_{v_0}={\bf m}(\Delta_{1,0})=\mu_0$, properties 3 and 4 for $n=0$
are trivial.

Let $n\in \N$, and let the tree ${\cal T}_{n-1}$,
the family of cubes $\Theta_{n-1}$ and the mapping
$F_{n-1}$ satisfying the conditions 1--4 be defined.
Construct ${\cal T}_n$, $\Theta _n$ and $F_n$.

If $\{\Delta _{i,n}\}_{i=1}^{r_n}\subset \Theta_{n-1}$, then
put ${\cal T}_n={\cal T}_{n-1}$, $\Theta_n=\Theta_{n-1}$,
$F_n=F_{n-1}$. Let
$$
I_n:=\{i=1, \, \dots , \, r_n:\Delta _{i,n}\notin
\Theta_{n-1}\}\ne \varnothing .
$$
Consider $i\in I_n$. Denote by $x_{i,n}$ the center of the cube $\Delta
_{i,n}$ and set $\tau _{i,n}=T_0(x_{i,n})$, $\gamma _{i,n}=\gamma
_{x_{i,n}}$,
$$
t_{i,n}=\min \{t\in [0, \, \tau _{i,n}]:\gamma _{i,n}(t) \in
\overline{\Omega}_{{\cal T}_{n-1},F_{n-1}}\}.
$$
Then
\begin{align}
\label{tin} t_{i,n}<\tau_{i,n}.
\end{align}

Let $k=k(i)$, $\{Q_1^i, \, \dots , \, Q_k^i\}\subset
\Theta(\Omega)\backslash \Theta_{n-1}$ be the set of all cubes such that
$\gamma _{i,n}([0, \, t_{i,n}])\cap Q_j^i\ne
\varnothing$. Then
\begin{align}
\label{mqijge}
{\bf m}(Q_j^i)\ge \mu _n
\end{align}
by the property 1 of the tree ${\cal T}_{n-1}$ and by (\ref{mu0mu1}). On the other hand,
there exists $c_1(d, \, a)\in \N$ such that
\begin{align}
\label{mqijle} {\bf m}(Q_j^i)\le \mu _n+c_1(d, \, a).
\end{align}
Indeed, if $Q_j^i=\Delta _{i,n}$, then it follows from
the definition of $\mu_n$. If $Q_j^i\ne\Delta _{i,n}$, then for any
$t$ such that $\gamma _{i,n}(t)\in Q_j^i$ the inequality
$|x_{i,n}-\gamma_{i,n}(t)|\ge 2^{-\mu_n-1}$ holds. Therefore,
$t\underset{d}{\gtrsim}2^{-\mu _n}$. It remains to apply the
property c) of the polygonal arc $\gamma _{i,n}$ and Theorem \ref{whitney}.

Construct the sequence of different cubes
$Q_{j_1}^i, \, \dots , \, Q_{j_\nu}^i$, $\nu =\nu(i)$,
$1\le j_s\le k$, such that $Q_{j_1}^i=\Delta _{i,n}$,
$\gamma _{i,n}(t_{i,n})\in Q_{j_\nu}^i$ and $\dim (Q_{j_s}^i\cap Q^i_{j_{s+1}})=d-1$,
as well as the sequence $\tilde t_1<\dots <\tilde t_\nu$ such that
$\gamma_{i,n}(\tilde t_s)\in Q^i_{j_s}$, $1\le s\le \nu$.
Set
$$
\tilde t_0=0, \tilde t_1=\max \{t\in [0, \, t_{i,n}]:
\gamma _{i,n}(t)\in Q^i_{j_1}\}.
$$
Let the cubes $Q_{j_1}^i=\Delta _{i,n}$, $Q^i_{j_2}, \, \dots , \, Q^i_{j_s}$ and
numbers $\tilde t_0<\tilde t_1<\dots <\tilde t_s$ be constructed, with
\begin{align}
\label{til_t_sigma} \tilde t_\sigma=\max \{t\in [\tilde
t_{\sigma-1}, \, t_{i,n}]: \gamma _{i,n}(t)\in Q^i_{j_\sigma}\},
\;\; 1\le \sigma \le s.
\end{align}
If $\tilde t_s=t_{i,n}$, then the construction is completed. Suppose that
$\tilde t_s<t_{i,n}$. Denote by $J_s$ the set of indices $j'\in \{1, \, \dots , \, k\}
\backslash \{j_1, \, \dots , \, j_s\}$ such that
$$
\dim (Q_{j_s}^i\cap Q_{j'}^i)=d-1.
$$
Prove that $J_s\ne \varnothing$. Indeed, the property d) of the
polygonal arc $\gamma_{i,n}$ and (\ref{til_t_sigma}) imply that
$\gamma_{i,n}(\tilde t_s)\in {\rm int}_{d-1}(Q_{j_s}^i\cap Q)$ for
some $Q\in \Theta(\Omega)$, $Q\ne Q_{j_s}^i$. Moreover, $Q\notin
\cup _{\sigma =1}^{s-1}Q^i_{j_\sigma}$ by (\ref{til_t_sigma}).
Finally, from the definition of $t_{i,n}$ and from the condition
$\tilde t_s<t_{i,n}$ follows that $Q\notin \Theta_{n-1}$. Hence,
$Q\in \{Q^i_1, \, \dots, \, Q^i_k\}\backslash \{Q^i_{j_1}, \,
\dots, \, Q^i_{j_s}\}$.

Set $ \tilde t_{s+1}=\max\{t\in [\tilde t_s, \,
t_{i,n}]:\gamma _{i,n}(t) \in \cup _{j'\in J_s}Q^i_{j'}\}$. Define
$j_{s+1}\in J_s$ by the inclusion $\gamma _{i,n}(\tilde
t_{s+1})\in Q^i_{j_{s+1}}$. Show that $j_{s+1}$ is well-defined. Actually,
let $\gamma _{i,n}(\tilde t_{s+1})\in
Q_{j'}^i\cap Q_{j''}^i$, $j'$, $j''\in J_s$, $j'\ne j''$. The
property d) of $\gamma _{i,n}$ implies that $\gamma
_{i,n}(\tilde t_{s+1})\in {\rm int}_{d-1}(Q_{j'}^i\cap
Q_{j''}^i)$. By (\ref{tin}), we have $\tilde t_{s+1}<\tau_{i,n}$.
Therefore, there exists $\delta \in (0, \, \tau_{i,n}-\tilde
t_{s+1})$ such that $\gamma _{i,n}(\tilde t_{s+1}+\delta) \in {\rm
int}\, Q_{j'}^i\cup {\rm int}\, Q_{j''}^i$ (it follows again from
the property d) of $\gamma _{i,n}$). Hence, $\tilde
t_{s+1}<t_{i,n}$, and we get the contradiction with the definition of $\tilde
t_{s+1}$. Notice that $\tilde t_{s+1}=\max\{t\in [\tilde
t_s, \, t_{i,n}]:\gamma _{i,n}(t) \in Q^i_{j_{s+1}}\}$.

For each $s\in \overline{1, \, \nu-1}$ denote by $\tilde J_s$
the set of indices $j'\in \{1, \, \dots , \, k\}\backslash \{j_1, \, \dots
, \, j_{s+1}\}$ such that $Q_{j_{s+1}}^i\cap Q_{j'}^i\ne
\varnothing$. By (\ref{mqijge}) and (\ref{mqijle}) (or by
Theorem \ref{whitney}), there exists $c_2(d, \, a)\in \N$ such that
${\rm card}\, \tilde J_s\le c_2(d, \, a)$. Let us prove that for any
$s\in \N$ such that $\tilde t_{s+c_2(d, \, a)+2}<t_{i,n}$
the following inequality holds:
\begin{align}
\label{tijc21} \tilde t_{s+c_2(d, \, a)+3}-\tilde t_{s+1}\ge 2^{-\mu
_n-c_1(d, \, a)}.
\end{align}
Indeed, let $j_\sigma\in \tilde J_s$ for any $\sigma\in \{s+2, \,
\dots, \, s+c_2(d, \, a)+3\}$. Since all indices $j_\sigma$ are different,
we have ${\rm card}\, \tilde J_s\ge c_2(d, \, a)+1$, which leads
to a contradiction. Assume that there exists
$\sigma\in \{s+2, \, \dots, \, s+c_2(d, \, a)+3\}$ such that $j_\sigma \notin \tilde J_s$.
Then $Q^i_{j_{s+1}}\cap Q^i_{j_\sigma}=\varnothing$.
Since $\gamma_{i,n}(\tilde t_{s+1})\in Q^i_{j_{s+1}}$ and
$\gamma_{i,n}(\tilde t_\sigma)\in Q^i_{j_\sigma}$, then
$\tilde t_\sigma-\tilde t_{s+1}\ge |\gamma_{i,n}(\tilde t_\sigma)-
\gamma_{i,n}(\tilde t_{s+1})|\stackrel{(\ref{mqijle})}{\ge}
2^{-\mu_n-c_1(d, \, a)}$, which implies (\ref{tijc21}).

On the other hand, since $\gamma _{i,n}(t_{i,n})\in Q^i_{j_\nu}$,
then by (\ref{mqijge}) and Theorem \ref{whitney} there exists
$c_3(d)\in \N$ such that ${\rm dist}\, (\gamma _{i,n}(t_{i,n}),
\, \partial \Omega)\le 2^{-\mu _n+c_3(d)}$. Therefore,
$t_{i,n}\le 2^{-\mu _n +c_4(d, \, a)}$ for some $c_4(d, \, a)\in \N$ (see the property c) of
$\gamma _{i,n}$). This together with (\ref{tijc21}) yield that $\nu\le
c_5(d, \, a)$ for some $c_5(d, \, a)>0$.

Set $\Theta^{I_n}=\{Q^i_{j_s}, \; i\in I_n, \; 1\le s\le
\nu(i)\}$. Construct a directed graph ${\cal G}$ (without multiple
edges and loops) and a one-to-one mapping $\Phi :{\bf V}({\cal G})
\rightarrow \Theta^{I_n}$: we regard the vertices $v'$, $v''$ as
adjacent, regard $v''$ as the head of $(v'', \, v')$ and regard
$v'$ as the tail of $(v'', \, v')$ if and only if there exists
$i\in I_n$ and $s\in 1, \, \dots , \, \nu(i)-1$ such that
$\Phi(v')=Q^i_{j_s}$, $\Phi(v'')=Q^i_{j_{s+1}}$. Since
$\dim\left(Q^i_{j_s}\cap Q^i_{j_{s+1}}\right)=d-1$, then $\Phi$ is
consistent with the structure of ${\cal G}$.

Let $\{\Phi ^{-1}(Q^i_{j_{\nu(i)}})\}_{i\in I_n}=\{w_1, \, \,
\dots, \, w_{r'_n}\}$. Add to the graph ${\cal G}$ a vertex $\hat
v$ and connect it with vertices $w_{i'}$, $1\le i'\le r'_n$,
considering that $\hat v$ is the head of these edges. Thereby, we
obtain a directed graph $\hat{\cal G}$ such that $(\hat{\cal G},
\, \hat v) \in \mathfrak{G}_{c_5(d, \, a)}$ (see Definition
\ref{dgk}). By Lemma \ref{graph_tree}, there exists a tree ${\cal
T}'$ rooted at $\hat v$, which is a subgraph of $\hat{\cal G}$
such that ${\bf V}({\cal T}')={\bf V}(\hat{\cal G})$ and for any
$v\in {\bf V}({\cal T}')$ the following inequality holds:
\begin{align}
\label{rohvv} \rho(\hat v, \, v) \le c_5(d, \, a).
\end{align}

Since $\gamma _{i,n}(t_{i,n})\in Q^i_{j_{\nu(i)}}=\Phi(w_{i'})$ for some
$i'\in \{1, \, \dots, \, r_n'\}$, then the definition of
$t_{i,n}$ and the property d) of $\gamma _{i,n}$ yield that
there exists a vertex $u_{i'}\in {\bf V}({\cal T}_{n-1})$ and a cube
$\tilde{Q}_{i'}=F_{n-1}(u_{i'})$ such that $\dim (\tilde{Q}_{i'}\cap
\Phi(w_{i'}))=d-1$. Put
$$
{\cal T}_n={\bf J}({\cal T}_{n-1}, \, ({\cal T}')_{w_1}, \, \dots ,
\, ({\cal T}')_{w_{r_n'}}; u_1, \, w_1, \, \dots , \, u_{r_n'}, \,
w_{r_n'}),
$$
$$
F_n(v)=\left\{ \begin{array}{l} F_{n-1}(v), \, \text{ if }
v\in {\cal T}_{n-1}, \\
\Phi(v), \; \text{ if }v\in {\bf V}(({\cal T}')_{w_i}), \;\;
1\le i\le r'_n,
\end{array}\right .
$$
$\Theta_n=F_n({\bf V}({\cal T}_n))$. Then the mapping $F_n$
is bijective and consistent with the structure of ${\cal T}_n$ (it
follows from the construction, from the definition of the graph ${\cal G}$ and
from the induction assumption).

The properties 1--2 of ${\cal T}_n$ follow from the construction,
(\ref{mu0mu1}) and (\ref{mqijle}), the property 3 follows from (\ref{rohvv}).
Check the property 4. Let $v'$, $v''\in {\bf V}({\cal T}_n)$,
$v'>v''$. If $v'$, $v''\in {\bf V}({\cal T}_{n-1})$, then it holds by
the induction assumption. Let $v'\in {\bf
V}({\cal T}_n)\backslash {\bf V}({\cal T}_{n-1})$. If for any
$v\in [v'', \, v']$ the inequality $m_v\ge \mu _n$ holds, then $v''\in {\bf
V}({\cal T}_n)\backslash {\bf V}({\cal T}_{n-c_*(d,a)-2})$. Indeed,
if $v\in {\bf V}({\cal T}_{n-j})$ and $m_v\ge \mu _n$, then
$$
\mu _{n-j}+j-1\stackrel{(\ref{mu0mu1})}{\le} \mu
_{n-j+1}+j-1\stackrel{(\ref{mu0mu1})}{\le} \mu
_{n-j+2}+j-2\stackrel{(\ref{mu0mu1})}{\le} \dots
\stackrel{(\ref{mu0mu1})}{\le} \mu _n\le m_v\le \mu _{n-j}+c_*(d,
\, a)
$$
(the last inequality follows from the property 2 of the tree ${\cal T}_{n-j}$),
which implies $j\le 1+c_*(d, \, a)$. Hence, by the property 3 of the trees
${\cal T}_k$, $1\le k\le n$, we have
\begin{align}
\label{rovvstrih} \rho(v', \, v'')\le (c_{**}(d, \, a)+1)(c_*(d,
\, a)+2);
\end{align}
by the property 2 of ${\cal T}_n$, the inequality $m_{v'}-m_{v''}\ge
\mu_n-(\mu_n+c_*(d, \, a))=-c_*(d, \, a)$ holds, so
$$
l_*(m_{v'}-m_{v''})\ge -l_*c_*(d, \, a)\stackrel{(\ref{k_st})}{=}
(c_{**}(d, \, a)+1)(c_*(d, \, a)+2)-k_*\ge \rho(v', \, v'')-k_*.
$$

Assume now that $v'\in {\bf V}({\cal T}_n)\backslash {\bf V}({\cal T}_{n-1})$ and
$W:=\{v\in [v'', \, v']:m_v\le \mu _n-1\}\ne \varnothing$. Set
$v_*=\max W$. Then it follows from (\ref{mqijge}) that
\begin{align}
\label{vstn1}
v_*\in {\bf V}({\cal T}_{n-1}), \;\;\; v_*<v'.
\end{align}
Denote by $v_{**}$ the vertex in $[v_*, \, v']$ that is the successor of
$v_*$. By the inequality (\ref{rovvstrih})
applied to $v'':=v_{**}$, by (\ref{vstn1}) and by
the induction assumption we have
$$
l_*(m_{v'}-m_{v''})=l_*(m_{v'}-m_{v_*})+l_*(m_{v_*}-m_{v''})\ge
l_*+\rho(v_*, \, v'')-k_*=
$$
$$
=(l_*-1-\rho(v', \, v_{**}))+\rho(v', \, v'')-k_*\ge
$$
$$
\ge(l_*-1-(c_{**}(d, \, a)+1)(c_*(d, \, a)+2))+\rho(v', \,
v'')-k_*\stackrel{(\ref{l_st})}{=} \rho(v', \, v'')-k_*.
$$

It remains to take as ${\cal T}$ the tree which is obtained by countable
repeating of induction steps, and to define $F$ by $F|_{{\cal T}_n}=F_n$.
\end{proof}
\begin{Cor}
\label{cor_omega_t} Let $\Omega\subset (0, \, 1)^d$, $\Omega\in
{\bf FC}(a)$, let ${\cal T}$ and $F:{\bf V}({\cal T})\rightarrow
\Theta(\Omega)$ be the tree and the mapping constructed in Lemma
\ref{constr_omega_t}. Then for any subtree ${\cal T}'$ in ${\cal T}$
we have
\begin{align}
\label{fghjkl}
\Omega_{{\cal T}',F}\in {\bf FC}(b_*), \text{ where }b_*=b_*(a, \, d)>0.
\end{align}
Here
\begin{align}
\label{kratn} {\rm card}\, {\bf V}_1(v)\underset{d}{\lesssim}1,
\;\; v\in {\bf V}({\cal T})
\end{align}
(see the notation on the page \pageref{v1v}),
\begin{align}
\label{vol1} {\rm mes}\, \Omega_{{\cal
T'},F}\underset{a,d}{\asymp}{\rm mes}\, F(v), \text{ where }v\text{
is the minimal vertex in }{\cal T}'.
\end{align}
\end{Cor}
\begin{proof}
Indeed, Lemma \ref{omega_t_cone} implies that
$$\Omega_{{\cal T}',F}\in {\bf FC}(\hat a(k_*(a, \, d), \, l_*(a, \, d), \, d));$$
(\ref{kratn}) and (\ref{vol1}) follow from (\ref{card_sopr}) and
(\ref{vol}), respectively.
\end{proof}

Let ${\cal T}_1, \, \dots , \, {\cal T}_l\in {\bf ST}({\cal
T})$, $\cup _{j=1}^l{\cal T}_j={\cal T}$ and ${\cal T}_i\cap {\cal
T}_j =\varnothing$ for any $i\ne j$. Then we call $\{{\cal T}_1, \,
\dots , \, {\cal T}_l\}$ a partition of the tree ${\cal
T}$. If $\mathfrak{S}$ is a partition of ${\cal T}$ and
${\cal A}\in {\bf ST}({\cal T})$, then put
$$
\mathfrak{S}|_{\cal A}=\{{\cal A}\cap {\cal T}':
{\cal T}'\in \mathfrak{S}, \; {\cal A}\cap {\cal T}'
\ne \varnothing\}.
$$

Let ${\cal T}$ be a tree, and let $\Psi :2^{{\bf V}(\cal T)}\rightarrow \R_+$.
Throughout, we denote $\Psi({\cal T}'):=\Psi({\bf V}({\cal T}'))$ for
${\cal T}'\in {\bf ST}({\cal T})$.

\begin{Lem}
\label{lemma_pro_sigma_t_g}
Let $({\cal T}, \, v_*)$ be a tree, and let
$\Psi :2^{{\bf V}(\cal T)}\rightarrow \R_+$ satisfy the following conditions:
\begin{align}
\label{prop_psi}
\Psi(V_1\cup V_2)\ge \Psi(V_1)+\Psi(V_2), \; V_1, \, V_2\subset
{\bf V}({\cal T}), \;\; V_1\cap V_2=\varnothing;
\end{align}
\begin{align}
\label{prop_psi1}
\text{if }\{v_n\}_{n\in \N}\subset {\bf V}({\cal T}),\;\;
v_1<\dots <v_n<\dots, \text{ then }\lim \limits _{n\rightarrow \infty}
\Psi({\cal T}_{v_n})=0.
\end{align}
Let $\gamma >0$, $\Psi({\cal T})>2\gamma$. Then there exists a unique vertex
$\hat v\in {\bf V}({\cal T})$ such that
\begin{align}
\label{pbfvtslbtvpv1}
\Psi({\cal T}_{\hat v})>\Psi({\cal T})-\gamma
\end{align}
and for any $v'\in {\bf V}_1(\hat v)$
\begin{align}
\label{tvsltg}
\Psi({\cal T}_{v'})\le \Psi({\cal T})-\gamma .
\end{align}
\end{Lem}
\begin{proof}
Denote
$$
E=\{v\in {\bf V}({\cal T}):\Psi({\cal T}_v)>\Psi({\cal T})-\gamma\}.
$$
Since $v_*\in E$, then $E\ne \varnothing$.
Show that $E$ is a finite chain. At first check that any two vertices in
$E$ are comparable. Indeed, let $v'$, $v''\in E$ be incomparable.
Then ${\cal T}_{v'}\cap {\cal T}_{v''}=\varnothing$ and
$$
\Psi({\cal T})\stackrel{(\ref{prop_psi})}{\ge}
\Psi({\cal T}_{v'})+\Psi({\cal T}_{v''})> 2\Psi({\cal T})-2\gamma ,
$$
i.e. $\Psi({\cal T})<2\gamma$. This contradicts the hypothesis of
Lemma. Therefore, $E$ is a chain. Prove that $E$ is finite. Indeed,
otherwise there exists a sequence $\{v_n\}_{n\in
\N}\subset E$ such that $v_1<\dots <v_n<\dots$ and $\Psi({\cal
T}_{v_n})>\Psi({\cal T})-\gamma\ge \gamma$. This contradicts to
(\ref{prop_psi1}).

As $\hat v$ we take the maximal vertex in $E$.
Since $E$ is a chain, the vertex satisfying
(\ref{pbfvtslbtvpv1}) and (\ref{tvsltg})
is unique.
\end{proof}
Let the conditions of Lemma \ref{lemma_pro_sigma_t_g} hold, let $\hat v$
be a vertex satisfying (\ref{pbfvtslbtvpv1}) and (\ref{tvsltg}).
Define the partition $\Sigma({\cal T}, \, \gamma)$ of the tree ${\cal
T}$ by
\begin{align}
\label{def_sigma_t_g} \Sigma({\cal T}, \, \gamma)=\left\{{\cal
T}\backslash {\cal T}_{\hat v}\right\}\cup \{\hat v\} \cup \{{\cal
T}_{v'}\}_{v'\in {\bf V}_1(\hat v)}.
\end{align}

Let $x\in \R$. Denote by $\lceil x\rceil$ the nearest integer to $x$
from above.

In the following Lemma we construct a special partition of a
tree. Notice that a similar partition of a metric tree was constructed in
\cite{solomyak}.
\begin{Lem}
\label{lemma_o_razb_dereva}
Let $k\in \N$. Then for any tree ${\cal T}$
rooted at $v_*$ such that
\begin{align}
\label{cardvvvk}
{\rm card}\, {\bf V}_1(v)\le k, \;\; v\in {\bf V}({\cal T}),
\end{align}
for any mapping $\Psi :2^{{\bf V}(\cal T)}\rightarrow \R_+$
satisfying (\ref{prop_psi}) and (\ref{prop_psi1}) and
for any $\gamma >0$ there exists a partition $\mathfrak{S}({\cal T}, \, \gamma)$
of ${\cal T}$ with the following properties:
\begin{enumerate}
\item let $u\in {\bf V}({\cal T})$, $\mathfrak{S}({\cal T}_u, \, \gamma)\subset
\mathfrak{S}({\cal T}, \, \gamma)$; if $\Psi({\cal T}_u)>
(k+1)\gamma$,
$$
\Sigma({\cal T}_u, \, \gamma)=\left\{{\cal T}_u\backslash {\cal
T}_{\hat v_u}\right\}\cup\{\hat v_u\} \cup \{{\cal
T}_{v'}\}_{v'\in {\bf V}_1(\hat v_u)},
$$
then
\begin{align}
\label{ind_s_t_g} \mathfrak{S}({\cal T}_u, \, \gamma)=\left\{{\cal
T}_u\backslash {\cal T}_{\hat v_u}\right\}\cup\{\hat v_u\} \bigcup
\left(\bigcup _{v'\in {\bf V}_1(\hat v_u)} \mathfrak{S}({\cal
T}_{v'}, \, \gamma)\right)
\end{align}
with $\Psi({\cal T}_{\hat v_u})>\Psi({\cal T}_u)-\gamma$,
$\Psi({\cal T}_u\backslash {\cal T}_{\hat v_u})<\gamma$; if
$\Psi({\cal T}_u)\le(k+1)\gamma$, then $\mathfrak{S}({\cal T}_u, \,
\gamma)=\{{\cal T}_u\}$;
\item if ${\cal T}'\in \mathfrak{S}({\cal T}, \, \gamma)$ and
$\Psi({\cal T}')>(k+2)\gamma$, then ${\rm card}\, {\bf V}({\cal T}')=1$;
\item if $\left\lceil\frac{\Psi({\cal T})}{\gamma}\right\rceil \ge k+2$, then
${\rm card}\, \mathfrak{S}({\cal T}, \, \gamma)\le
(k+2)\left\lceil\frac{\Psi({\cal T})}{\gamma}\right\rceil
-(k+1)(k+2)$, otherwise ${\rm card}\, \mathfrak{S}({\cal T}, \,
\gamma)=1$;
\item let $v>v_*$; then ${\rm card}\, \mathfrak{S}({\cal T}, \, \gamma)|_{{\cal T}_v}\le
(k+2)\left(\left\lceil\frac{\Psi({\cal T}_v)}{\gamma}\right\rceil
+1\right)$;
\item if ${\cal A}\in \mathfrak{S}({\cal T}, \, \gamma)$ and ${\rm card}\,
{\bf V}({\cal A})\ge 2$, then either ${\cal A}={\cal T}_v$ for
some $v\in {\bf V}({\cal T})$ or ${\cal A}={\cal
T}_v\backslash {\cal T}_w$ for some $v$, $w\in {\bf V}({\cal
T})$, $w>v$; here in the second case $\Psi({\cal A})<\gamma$ and $\Psi({\cal T}_w)
>\Psi({\cal T}_v)-\gamma$.
\end{enumerate}
\end{Lem}
\begin{proof}
Let $m\in \Z _+$, and let ${\cal T}$ be a tree. We write $({\cal T},
\, \Psi)\in \mathfrak{R}_{m,\gamma}$ if $(m-1)\gamma
<\Psi({\cal T})\le m\gamma$.

If $v\in {\bf V}({\cal T})$,
then set $\mu(v)=\mu(v, \, {\cal T})=\left\lceil \frac{\Psi({\cal T}_v)}
{\gamma}\right\rceil$. Therefore,
$({\cal T}_v, \, \Psi)\in \mathfrak{R}_{\mu(v) ,\gamma}$.

Construct by induction on $m\in \Z_+$ partitions $\mathfrak{S}({\cal
T}, \, \gamma)$ for all ${\cal T}$ such that $({\cal T}, \,
\Psi)\in \mathfrak{R}_{m,\gamma}$. If $m\le k+1$, then set
$\mathfrak{S}({\cal T}, \, \gamma)=\{{\cal T}\}$. Let
\begin{align}
\label{mgek1}
m\ge k+1,
\end{align}
and let partitions $\mathfrak{S}({\cal T}, \, \gamma)$ satisfying
properties 1--5 be constructed for all $({\cal T}, \, \Psi)\in
\mathfrak{R}_{m',\gamma}$, $m'\le m$. Construct partitions for all
$({\cal T}, \, \Psi)\in \mathfrak{R}_{m+1,\gamma}$. Here
\begin{align}
\label{m1gpt} m\gamma <\Psi({\cal T})\le (m+1)\gamma .
\end{align}
Consider the partition $\Sigma({\cal T}, \, \gamma)$ defined by
(\ref{def_sigma_t_g}). Then
\begin{align}
\label{pbfvtslbtvpv}
\Psi({\cal T}\backslash{\cal T}_{\hat v})
\stackrel{(\ref{prop_psi})}{\le}
\Psi({\cal T})-\Psi({\cal T}_{\hat v})\stackrel{(\ref{pbfvtslbtvpv1})}{<}\gamma .
\end{align}
Since for any $v'\in {\bf V}_1(\hat v)$
\begin{align}
\label{mu1vvv} \Psi({\cal T}_{v'})\stackrel{(\ref{tvsltg})}{\le}
\Psi({\cal T})-\gamma \le m\gamma, \;\;\; \mu(v')=\mu(v', \, {\cal
T})\le m,
\end{align}
then by induction assumption the partition
$\mathfrak{S}({\cal T}_{v'}, \, \gamma)$ is defined. Set
\begin{align}
\label{mfs} \mathfrak{S}({\cal T}, \, \gamma)=\left\{{\cal
T}\backslash {\cal T}_{\hat v}\right\}\cup\{\hat v\} \bigcup
\left(\bigcup _{v'\in {\bf V}_1(\hat v)} \mathfrak{S}({\cal
T}_{v'}, \, \gamma)\right).
\end{align}
On this induction step we get that in the case
$\Psi({\cal T}_{\hat v})\in (m\gamma, \, (m+1)\gamma]$
\begin{align}
\label{s_t_g} \text{if }\Sigma({\cal T}_{\hat v}, \, \gamma)=\{\hat
v\}\cup \{{\cal T}_{v'}\}_{v'\in {\bf V}_1(\hat v)}, \text{ then }
\mathfrak{S}({\cal T}_{\hat v}, \, \gamma)=\{\hat
v\} \bigcup \left(\bigcup _{v'\in {\bf V}_1(\hat v)}
\mathfrak{S}({\cal T}_{v'}, \, \gamma)\right),
\end{align}
\begin{align}
\label{llll}
\text{ if } {\cal T}_{\hat v}\backslash {\cal T}_w\in \Sigma
({\cal T}_{\hat v}, \, \gamma)\text{ for some }w>\hat v,
\text{ then }{\cal T}_{\hat v}\backslash {\cal T}_w\in \mathfrak{S}
({\cal T}_{\hat v}, \, \gamma).
\end{align}
By the induction assumption, for the case $\Psi({\cal T}_{\hat
v})\le m\gamma$ (\ref{s_t_g}) and (\ref{llll}) hold as well.

Prove the property 1. If $u=v_*$, then the assertion follows from
the construction and (\ref{pbfvtslbtvpv}). If $u>v_*$, then by
(\ref{mfs}) we get that ${\cal T}_u\subset {\cal T}_{v'}$ for some
$v'\in {\bf V}_1(\hat v)$ or $u=\hat v$. In the first case the
property 1 holds by the induction assumption. Consider the second
case. If $\Psi({\cal T}_{\hat v})\le (k+1)\gamma$, then
$\mathfrak{S}({\cal T}_{\hat v}, \, \gamma)= \{{\cal T}_{\hat
v}\}$ by construction (see the base of induction). Hence,
$\mathfrak{S}({\cal T}_{\hat v}, \, \gamma)
\stackrel{(\ref{mfs})}{\not \subset} \mathfrak{S}({\cal T}, \,
\gamma)$. Let $\Psi({\cal T}_{\hat v})>(k+1)\gamma$. Then
(\ref{llll}), (\ref{mfs}) and the inclusion $\mathfrak{S}({\cal
T}_{\hat v}, \, \gamma)= \mathfrak{S}({\cal T}_u, \,
\gamma)\subset \mathfrak{S}({\cal T}, \, \gamma)$ imply that
$\Sigma({\cal T}_{\hat v}, \, \gamma)=\{\hat v\}\cup \{{\cal
T}_{v'}\}_{v'\in {\bf V}_1(\hat v)}$. Therefore, (\ref{ind_s_t_g})
follows from (\ref{s_t_g}).

The property 2 follows from (\ref{pbfvtslbtvpv}), (\ref{mfs})
and the induction assumption.

Prove the property 3. By definition of $\mu(v')$,
$$
\sum \limits _{v'\in {\bf V}_1(\hat v)} \mu(v')\gamma \le\sum
\limits _{v'\in {\bf V}_1(\hat v)} (\Psi({\cal
T}_{v'})+\gamma)\stackrel{(\ref{prop_psi}), \,
(\ref{cardvvvk})}{\le} \Psi({\cal
T})+k\gamma\stackrel{(\ref{m1gpt})}{\le} (m+1)\gamma +k\gamma ,
$$
that is
\begin{align}
\label{slimvinv1vmu1vm1k}
\sum \limits _{v'\in {\bf V}_1(\hat v)} \mu(v')\le m+1+k.
\end{align}
Set ${\bf V}'=\{v'\in {\bf V}_1(\hat v):\mu(v')\ge k+2\}$,
${\bf V}''=\{v'\in {\bf V}_1(\hat v):\mu(v')< k+2\}$.
If ${\rm card}\, {\bf V}'\ge 2$, then by the induction assumption
$$
{\rm card}\, \mathfrak{S}({\cal T}, \, \gamma)\stackrel{(\ref{mfs})}{\le} 2+\sum \limits
_{v'\in {\bf V}'} ((k+2)\cdot \mu(v')-(k+1)(k+2))+ {\rm card}\,
{\bf V}''\stackrel{(\ref{slimvinv1vmu1vm1k})}{\le}
$$
$$
\le 2-2(k+1)(k+2)+(k+2)(m+1+k)+k=(k+2)(m+1)-(k+1)(k+2).
$$
If ${\rm card}\, {\bf V}'=1$, then by the induction assumption
$$
{\rm card}\, \mathfrak{S}({\cal T}, \, \gamma)\stackrel{(\ref{mu1vvv}),(\ref{mfs})}{\le}
2+(k+2)m-(k+1)(k+2)+k=(k+2)(m+1)-(k+1)(k+2).
$$
If ${\rm card}\, {\bf V}'=0$, then
$$
{\rm card}\, \mathfrak{S}({\cal T}, \, \gamma)\le 2+k=(k+2)(k+2)-(k+1)(k+2)
\stackrel{(\ref{mgek1})}{\le}(k+2)(m+1)-(k+1)(k+2).
$$

Prove the property 4. If $v>\hat v$, then ${\cal T}_v\subset {\cal
T}_{v'}$ for some $v'\in {\bf V}_1(\hat v)$, and the assertion
follows from the induction assumption. If $v$, $\hat v$ are incomparable,
then ${\cal T}_v\subset {\cal T}\backslash {\cal T}_{\hat v}$ and
${\rm card}\, \mathfrak{S}({\cal T}, \, \gamma)|_{{\cal T}_v}=1$.
If $v\le \hat v$, then ${\rm card}\, \mathfrak{S}({\cal T}, \,
\gamma)|_{{\cal T}_v}\le {\rm card}\, \mathfrak{S}({\cal T}, \,
\gamma)$, and the assertion follows from the property 3 (which is already proved) and
from inequalities $\Psi({\cal T}_v)\ge \Psi({\cal T}_{\hat
v})\stackrel{(\ref{pbfvtslbtvpv1})}{>} \Psi({\cal T})-\gamma$.

Prove the property 5. Since ${\rm card}\, {\bf V}({\cal A})>1$, then
by (\ref{mfs}) there are two alternatives: a) ${\cal A}\in
\mathfrak{S}({\cal T}_{v'}, \, \gamma)$ for some $v'\in {\bf
V}_1(\hat v)$ (then the property 5 holds by the induction assumption);
b) ${\cal A}={\cal T}\backslash {\cal T}_{\hat v}$
(then the property 5 follows from this equality and
(\ref{pbfvtslbtvpv})).
\end{proof}
\begin{Lem}
\label{kol_el_peres} Let conditions of Lemma
\ref{lemma_o_razb_dereva} hold. Then there exists $c(k)>0$ such that
for any $\gamma >0$ and
\begin{align}
\label{gammasgegamma2}
\gamma '\ge \frac{\gamma}{2}
\end{align}
each element of the partition $\mathfrak{S}({\cal T}, \, \gamma)$ intersects with
no more than $c(k)$ elements of $\mathfrak{S}({\cal T}, \, \gamma ')$.
\end{Lem}
\begin{proof}
Let $\tilde{\cal T}$ be an element of $\mathfrak{S}({\cal T}, \,
\gamma)$, let $l$ be the number of elements of $\mathfrak{S}({\cal T}, \,
\gamma ')$ that intersect with $\tilde{\cal T}$. Suppose that
$l>1$. Then $\tilde{\cal T}$ contains at least two
vertices. By Lemma \ref{lemma_o_razb_dereva} (see property 5), there are
two cases.

{\it Case 1.} Let $\tilde{\cal T}={\cal T}_v$ for some
$v\in {\cal T}$. Apply Lemma \ref{lemma_o_razb_dereva}. By
property 2, $\Psi(\tilde{\cal T})\le (k+2)\gamma$. Hence, by
properties 3 and 4,
$$
l={\rm card}\, \mathfrak{S}({\cal T}, \, \gamma ')|_{{\cal T}_v}
\le (k+2)\left(\left\lceil(k+2)\frac{\gamma}{\gamma
'}\right\rceil+1\right)
\stackrel{(\ref{gammasgegamma2})}{\le}(k+2)(2k+5).
$$

{\it Case 2.} Let $\tilde{\cal T}={\cal T}_v\backslash {\cal
T}_w$, $v$, $w\in {\bf V}({\cal T})$, $w>v$, and
\begin{align}
\label{tvtwg} \Psi({\cal T}_v\backslash {\cal T}_w)<\gamma.
\end{align}
Put
$$
\tilde E=\{u\in {\bf V}({\cal T}):{\cal T}_v\subset {\cal T}_u, \;\;
\mathfrak{S}({\cal T}_u, \, \gamma ')\subset \mathfrak{S}({\cal T}, \, \gamma ')\}.
$$
This set is nonempty, since it contains $v_*$. Further, $\tilde E$
is a chain. Indeed, if vertices $u_1$ and $u_2$ are incomparable, then
${\cal T}_{u_1}$ and ${\cal T}_{u_2}$ do not intersect. Finally,
$\rho(v_*, \, v)\ge \rho(v_*, \, u)$ for any $u\in \tilde E$.
Therefore, $\tilde E$ contains the maximal element $u_0$.

Since $l>1$, then by Lemma \ref{lemma_o_razb_dereva} (see the property 1),
there exists a vertex $\hat u\ge u_0$ such that
\begin{align}
\label{stu0gs} \mathfrak{S}({\cal T}_{u_0}, \, \gamma ')=\{{\cal
T}_{u_0} \backslash {\cal T}_{\hat u}\}\cup \{\hat u\}\bigcup
\left(\bigcup _{v'\in {\bf V}_1(\hat u)} \mathfrak{S}({\cal
T}_{v'}, \, \gamma ')\right).
\end{align}
Show that
\begin{align}
\label{vlehatu}
v\le \hat u.
\end{align}
Indeed, if $v$ and $\hat u$ are incomparable, then ${\cal
T}_v\subset {\cal T}_{u_0} \backslash {\cal T}_{\hat u}$, and
$$
l={\rm card}\, \mathfrak{S}({\cal T}, \, \gamma')|_{\tilde{\cal
T}}\le {\rm card}\, \mathfrak{S}({\cal T}, \, \gamma')|_{{\cal
T}_v}={\rm card}\, \mathfrak{S}({\cal T}_{u_0}, \,
\gamma')|_{{\cal T}_v}=1.
$$
If $v>\hat u$, then $v\in {\cal T}_{v'}$ for some $v'\in
{\bf V}_1(\hat u)$, that is $v'\in \tilde E$ by (\ref{stu0gs})
and $v'>u_0$. This contradicts the fact that $u_0=\max \tilde E$.

Let $u\in {\bf V}({\cal T})$, $\mathfrak{S}({\cal T}_u, \, \gamma ')
\subset \mathfrak{S}({\cal T}, \, \gamma ')$, $u\ge v$. Denote by $l_u$ the number
of elements of the partition $\mathfrak{S}({\cal T}_u, \, \gamma ')$ that intersect
with ${\cal T}_v\backslash {\cal T}_w$.

Set
\begin{align}
\label{def_u} {\cal U}=\{u\ge v: \; \mathfrak{S}({\cal T}_u, \,
\gamma ') \subset \mathfrak{S}({\cal T}, \, \gamma '), \; l_u>1\}.
\end{align}

{\bf Assertion $(\star)$}. There exists a mapping $\varphi:{\cal U}\rightarrow {\bf V}({\cal
T})$ such that $\varphi(u)>u$ for any $u\in {\cal U}$,
\begin{align}
\label{llvv1u} \mathfrak{S}({\cal T}_{\varphi(u)}, \, \gamma
')\subset \mathfrak{S}({\cal T}_{u}, \, \gamma '), \;\; \Psi({\cal
T}_{\varphi(u)}) \le \Psi({\cal T}_u)-\gamma ' \text{ and } l_u\le
k+1+l_{\varphi(u)},
\end{align}
and if $l_{\varphi(u)}>1$, then
\begin{align}
\label{vsv} \varphi(u)<w.
\end{align}

{\it Proof of Assertion $(\star)$}. By Lemma
\ref{lemma_o_razb_dereva} (see the property 1), for any $u\in {\cal
U}$ there exists a vertex $\tilde u\ge u$ such that
$$
\mathfrak{S}({\cal T}_u, \, \gamma ')=\{{\cal T}_u \backslash
{\cal T}_{\tilde u}\}\cup \{\tilde u\}\bigcup \left(\bigcup
_{v'\in {\bf V}_1(\tilde u)} \mathfrak{S}({\cal T}_{v'}, \, \gamma
')\right)
$$
with $\Psi({\cal T}_{v'})\le \Psi({\cal T}_u)-\gamma'$ for any
$v'\in {\bf V}_1(\tilde u)$. Hence, $l_u\le 2+\sum \limits
_{v'\in {\bf V}_1(\tilde u)}l_{v'}$. In order to satisfy
first two expressions in (\ref{llvv1u}), it is sufficient to choose
$\varphi(u)\in {\bf V}_1(\tilde u)$.

Prove that if $l_{v'}>1$, then $v'<w$. If $v'\ge w$, then
$l_{v'}=0$. If $v'$ and $w$ are incomparable, then ${\cal T}_{v'}\subset
{\cal T}_v\backslash {\cal T}_w$ (since
$v'>u\stackrel{(\ref{def_u})}{\ge} v$ and ${\cal T}_w\cap {\cal
T}_{v'}=\varnothing$). Then
$$
\Psi({\cal T}_{v'})\le\Psi({\cal T}_v\backslash {\cal
T}_w)\stackrel{(\ref{tvtwg})}{<}\gamma
\stackrel{(\ref{gammasgegamma2})}{\le}2\gamma '.
$$
Therefore, $l_{v'}\le {\rm card}\,
\mathfrak{S}({\cal T}_{v'}, \, \gamma')=1$ (see the property 3
in Lemma \ref{lemma_o_razb_dereva}).

Set $A_u=\{v' \in {\bf V}_1(\tilde u):l_{v'}>1\}$.  From ${\rm
card}\, \{v' \in {\bf V}_1(\tilde u):v'<w\}\le 1$ follows that
${\rm card}\, A_u\le 1$. If ${\rm card}\, A_u=1$, then we take as
$\varphi(u)$ an element of $A_u$ (then (\ref{vsv}) holds). If
$A_u=\varnothing$, then we take as $\varphi(u)$ an arbitrary
element of the set ${\bf V}_1(\tilde u)$. In both cases the last
inequality in (\ref{llvv1u}) holds. This completes the proof of
the Assertion ($\star$).

From (\ref{stu0gs}) follows that $\mathfrak{S}({\cal T}_{u'}, \,
\gamma')\subset \mathfrak{S}({\cal T}, \, \gamma')$ for any $u'\in
{\bf V}_1(\hat u)$ and $u'>\hat
u\stackrel{(\ref{vlehatu})}{\ge}v$. Hence, if $l_{u'}>1$, then
$u'\in {\cal U}$.

By (\ref{stu0gs}) and the condition $u_0\in \tilde E$, we have $l\le 2+\sum
\limits _{u'\in {\bf V}_1(\hat u)}l_{u'}$. Estimate $l_{u'}$ for each
$u'\in {\bf V}_1(\hat u)$ such that $l_{u'}>1$. For this we use (\ref{llvv1u}).
If $l_{\varphi(u')}\le 1$,
then $l_{u'}\le k+2$. Let $l_{\varphi(u')}>1$. Then
$\varphi(u')\in {\cal U}$. Applying (\ref{llvv1u}) once again,
we get $l_{u'}\le k+1+k+1+l_{\varphi(\varphi(u'))}$ with
$$
\Psi({\cal
T}_{\varphi(\varphi(u'))})\stackrel{(\ref{llvv1u})}{\le}
\Psi({\cal T}_{\varphi(u')})-\gamma
'\stackrel{(\ref{llvv1u})}{\le} \Psi({\cal T}_{u'})-2\gamma
'\stackrel{(\ref{gammasgegamma2})}{\le} \Psi({\cal T}_{u'})-
\gamma.
$$
If $l_{\varphi(\varphi(u'))}>1$, then
$\varphi(\varphi(u'))\stackrel{(\ref{vsv})}{<}w$, and $\Psi({\cal
T}_w) \le\Psi({\cal T}_{\varphi(\varphi(u'))})$. On the other hand,
by Lemma \ref{lemma_o_razb_dereva} (see the property 5), we have $\Psi({\cal
T}_w)> \Psi({\cal T}_v)-\gamma$. Taking into account the condition $u'>\hat
u\stackrel{(\ref{vlehatu})}{\ge} v$, we get the contradictory chain of inequalities
$$
\Psi({\cal T}_v)-\gamma <\Psi({\cal T}_w) \le\Psi({\cal
T}_{\varphi(\varphi(u'))})\le \Psi({\cal T}_{u'})-\gamma \le
\Psi({\cal T}_v)-\gamma .
$$
Hence, $l_{\varphi(\varphi(u'))}\le 1$, $l_{u'}\le 2k+3$ and $l\le
2+k(2k+3)$.
\end{proof}
\begin{Cor}
\label{lemma_o_razb_dereva1} Let conditions of Lemma
\ref{lemma_o_razb_dereva} hold, and let $\Psi({\cal T})>0$. Then there exists
a number $C(k)>0$ such that for any $n\in \N$ there exists
a partition $\mathfrak{S}_n=\mathfrak{S}({\cal T}, \, \Psi({\cal
T})/n)$ of the tree ${\cal T}$ into at most $C(k)n$ subtrees
${\cal T}_j$ such that $\Psi({\cal T}_j)\le \frac{(k+2)\Psi({\cal
T})}{n}$ for any $j$ that satisfies the condition ${\rm
card}\, {\bf V}({\cal T}_j)\ge 2$. In addition, there exists $C_1(k)>0$ such that
if $m\le 2n$, then each element $\mathfrak{S}_n$ intersects
with at most $C_1(k)$ elements of $\mathfrak{S}_m$.
\end{Cor}

{\bf The partition of a cube.} Let $\Phi$ be a nonnegative
function defined on Lebesgue measurable subsets of $\R^d$
and satisfying the conditions
\begin{align}
\label{pr1}
\Phi(A_1)+\Phi(A_2)\le \Phi(A_1\cup A_2), \ \ \mbox{ where }\ \ A_1, \, A_2\subset
\R^d \ \ \mbox{ do not overlap};
\end{align}
\begin{align}
\label{pr2} \mbox{if {\rm mes\,}} A_n\rightarrow 0, \mbox{ then }
\Phi(A_n)\rightarrow 0 \quad (n\rightarrow\infty).
\end{align}

Denote by ${\cal R}$ the family of sets $K
\backslash K '$, where $K \in {\cal K}$, $K '\in \Xi (K)$, $K'\ne
K$.

The following Lemma is proved in \cite{vas_mnogo}.
\begin{Lem}
\label{partition}
Let $K\in {\cal K}$, and let the function $\Phi$ satisfy
(\ref{pr1}) and (\ref{pr2}). Then for any $n\in \N$ there
exists a partition $T_n=T_n(K)$ of the cube $K$ with the following properties:
\begin{enumerate}
\item the number of elements of $T_n$ does not exceed $2^dn;$
\item for any $\Delta \in T_n$ the inequality $\Phi(\Delta)
\le 3n^{-1}\Phi(K)$ holds;
\item each element of $T_n$ belongs to $\Xi(K)$ or to ${\cal R}$;
\item there exists $C(d)\in\N$ such that for $l\le 2m$ $(l,\,m\in \N)$
each element of $T_m$ overlaps with at most $C(d)$
elements of $T_l$.
\end{enumerate}
\end{Lem}
A similar partition is constructed in \cite{cohen_devore}. As
proved, this partition satisfies properties similar to 1--3 (with
other constants estimating the number of elements and the value
$\Phi(\Delta)$).

{\bf Definition of functions $\Phi$ and $\Psi$.} Let $\alpha_1, \,
\dots, \, \alpha_{l_*}>0$, $\sum \limits _{j=1}^{l_*} \alpha_j=1$,
let $\mu_1, \, \dots, \, \mu_{l_*}$ be finite absolutely
continuous measures on $\Omega$. From the Radon -- Nikodym theorem
and from the absolute continuity of the Lebesgue integral follows
that
\begin{align}
\label{mu_prop} \mu_j(A)\rightarrow 0 \text{ as }{\rm mes}\,
A\rightarrow 0.
\end{align}
For a Lebesgue measurable set $A\subset \R^d$ we put
\begin{align}
\label{def_phi} \Phi(A)=\prod _{j=1}^{l_*} (\mu_j(A\cap
\Omega))^{\alpha_j}.
\end{align}
The conditions (\ref{mu_prop}) and $\alpha_j>0$ imply (\ref{pr2}).
Further, by the H\"{o}lder inequality we get
\begin{align}
\label{hol_cor}
\prod_{j=1}^{l_*} (b_j+c_j)^{\alpha_j}\ge
\prod_{j=1}^{l_*} b_j^{\alpha_j}+\prod_{j=1}^{l_*} c_j^{\alpha_j}, \;\;
b_j\ge 0, \; c_j\ge 0,
\end{align}
which yields (\ref{pr1}).

Let ${\cal T}$ and $F$ be the tree and the mapping constructed in
Lemma \ref{constr_omega_t}. Define the mapping
$\Psi:2^{{\bf V}({\cal T})}\rightarrow \R_+$ by formula
\begin{align}
\label{def_psi}
\Psi({\bf W})=\Phi\left(\cup_{v\in {\bf W}}F(v)\right), \;\;\; {\bf W}\subset
{\bf V}({\cal T}).
\end{align}
Then (\ref{pr1}) implies (\ref{prop_psi}). Prove
(\ref{prop_psi1}). Let $\{v_j\}_{j\in \N}\subset {\bf V}({\cal
T})$, $v_1<\dots<v_n<\dots$. From (\ref{mvsmvkvskvms}) follows
that $m_{v_n}\underset{n\to \infty}{\to} \infty$. Hence, ${\rm
mes}\, F(v_n)\underset{n\to\infty}{\to}0$, and by (\ref{vol1}) we
get ${\rm mes}\, \Omega _{{\cal T}_{v_n},F} \underset{n\to\infty}
{\to}0$. Therefore, (\ref{prop_psi1}) follows from (\ref{pr2}).

\begin{Lem}
\label{approx} For any $n\in \N$ there exists a family of partitions
$\{\mathcal{B}_{n,m}\}_{m\in \Z_+}$ of the domain $\Omega$ with the following properties:
\begin{enumerate}
\item ${\rm card}\, \mathcal{B}_{n,m}\underset{d}{\lesssim}2^mn$;

\item if $E\in \mathcal{B}_{n,m}$, then
\begin{enumerate}
\item either $E=\Omega_{{\cal T}',F}$ for some subtree ${\cal T}'\subset
{\cal T}$ or $E\subset F(w)$ for some vertex $w\in {\bf
V}({\cal T})$ and $E\in \Xi(F(w))\cup {\cal R}$;
\item $\Phi(E)\underset{d}{\lesssim}\frac{\Phi(\Omega)}{2^mn}$;
\end{enumerate}
\item there exists $C_*(d)$ such that each element of
$\mathcal{B}_{n,m}$ overlaps with at most $C_*(d)$
elements of $\mathcal{B}_{n,m\pm 1}$.
\end{enumerate}
\end{Lem}
\begin{proof}
We suppose that $\Phi(\Omega)>0$ (otherwise set
$\mathcal{B}_{n,m}=\{\Omega_{{\cal T},F}\}$). From (\ref{kratn})
follows that for any vertex $v\in {\bf V}({\cal T})$ we have ${\rm
card}\, {\bf V}_1({\cal T})\underset{d}{\lesssim} 1$. By Corollary
\ref{lemma_o_razb_dereva1}, for any $n\in \N$, $m\in \Z_+$ there
exists a partition $\mathfrak{S}_{2^mn}=\{{\cal
T}_j^{m,n}\}_{j=1}^{j_*(m,n)}$ of ${\cal T}$ with the following
properties:
\begin{enumerate}
\item ${\cal T}_j^{m,n}$ is a tree;
\item $j_*(m,\, n)\underset{d}{\lesssim} 2^mn$;
\item if ${\rm card}\, {\bf V}({\cal T}_j^{m,n})\ge 2$, then
$\Psi({\cal T}_j^{m,n})\underset{d}{\lesssim}\frac{\Psi({\cal
T})}{2^mn}$;
\item there exists $\hat C(d)>0$ such that each element of
$\mathfrak{S}_{2^mn}$ intersects with at most $\hat C(d)$
elements of $\mathfrak{S}_{2^{m\pm 1}n}$.
\end{enumerate}

Denote
\begin{align}
\label{def_j_m_n} J_{m,n}=\left\{j\in \overline{1, \, j_*(m,\,
n)}:\, {\rm card}\, {\bf V}({\cal T}_j^{m,n})=1, \;\; \Psi({\cal
T}_j^{m,n})\ge \frac{\Psi({\cal T})}{2^mn}\right\}.
\end{align}
Let $j\in J_{m,n}$, ${\bf V}({\cal T}_j^{m,n})=\{v_j^{m,n}\}$.
Put $\Delta_j^{m,n}:=F(\{v_j^{m,n}\})$,
\begin{align}
\label{def_l_j_m_n} l_{j,m,n}=\left\lceil
\frac{2^mn\Phi(\Delta_j^{m,n})}{\Phi(\Omega)}\right\rceil.
\end{align}
Then
\begin{align}
\label{sum_njm} \sum \limits _{j\in J_{m,n}} l_{j,m,n}\le j_*(m,\,
n)+\sum \limits _{j\in J_{m,n}}
\frac{2^mn\Phi(\Delta_j^{m,n})}{\Phi(\Omega)}\stackrel{(\ref{pr1})}{\le}
j_*(m,\, n)+2^mn \underset{d}{\lesssim} 2^mn.
\end{align}

Let $T_{l_{j,m,n}}(\Delta_j^{m,n})$ be the partition of the cube
$\Delta_j^{m,n}$ defined in Lemma
\ref{partition}. Set
$$
{\cal B}_{n,m}=\{\Omega _{{\cal T}_j^{m,n},F}\}_{j\notin
J_{m,n}}\bigcup \left(\bigcup _{j\in
J_{m,n}}T_{l_{j,m,n}}(\Delta_j^{m,n})\right).
$$

Check the property 1:
$$
{\rm card}\, {\cal B}_{n,m}\le j_*(m,\, n)+\sum \limits _{j\in
J_{m,n}}2^d l_{j,m,n} \stackrel{(\ref{sum_njm})}
{\underset{d}{\lesssim}} 2^mn.
$$
Check the property 2. Item a) follows from the construction and
item 3 of Lemma \ref{partition}. Check item b). Let $E=\Omega
_{{\cal T}_j^{m,n},F}$ with $j\notin J_{m,n}$. From the property 3
of $\mathfrak{S}_{2^mn}$ and from the definition of $J_{m,n}$
follows that
\begin{align}
\label{phi_e_phi} \Phi(E)\stackrel{(\ref{def_dom_by_tree})}{=}
\Phi\left(\cup _{v\in {\bf V}({\cal T}_j^{m,n})} F(v)\right)
\stackrel{(\ref{def_psi})}{=}\Psi({\cal T}_j^{m,n})\underset{d}{\lesssim} \frac{\Psi({\cal
T})}{2^mn}\stackrel{(\ref{def_psi})}{=}\frac{\Phi(\Omega)}{2^mn}.
\end{align}
Let $E\in T_{l_{j,m,n}}(\Delta_j^{m,n})$ for some $j\in J_{m,n}$.
From item 2 of Lemma \ref{partition} follows that
$$
\Phi(E)\le
3l_{j,m,n}^{-1}\Phi(\Delta_j^{m,n})\stackrel{(\ref{def_l_j_m_n})}{\le}
\frac{3\Phi(\Omega)}{2^mn}.
$$
Check the property 3. Let $E=\Omega_{{\cal T}_j^{m,n},F}$, $j\notin J_{m,n}$.
Denote
\begin{align}
\label{ij_pm} I_j^{\pm}=\{i\in J_{m\pm 1,n}:\; v_i^{m\pm 1,n}\in
{\bf V}({\cal T}_j^{m,n})\}.
\end{align}
Then
$$
{\rm card}\, I_j^{\pm}\cdot \frac{\Psi({\cal
T})}{2^mn}\stackrel{(\ref{def_j_m_n})}{\le} 2\sum \limits _{i\in
I_j^{\pm}}\Psi(\{v_i^{m\pm 1,n}\}) \stackrel
{(\ref{prop_psi})}{\le} 2\Psi({\cal T}_j^{m,n})
\stackrel{(\ref{phi_e_phi})}{\underset{d}{\lesssim}}\frac{\Psi({\cal
T})}{2^mn},
$$
which implies
\begin{align}
\label{cijpm} {\rm card}\, I_j^{\pm}\underset{d}{\lesssim} 1.
\end{align}
Therefore, by item 1 of Lemma \ref{partition} and by the property
4 of $\mathfrak{S}_{2^mn}$, we have
$$
{\rm card}\, \{E'\in {\cal B}_{n,m\pm 1}:\; ({\rm int}\, E')\cap
({\rm int}\, E)\ne \varnothing \}\le \hat C(d)+\sum \limits _{i\in
I_j^{\pm}}2^d l_{i,m\pm 1,n}\stackrel{(\ref{def_l_j_m_n})}{\le}
$$
$$
\le \hat C(d)+2^d\sum \limits _{i\in I_j^{\pm}}\left(\frac{2^{m\pm
1}n\Phi(\Delta_i^{m\pm 1,n})}{\Phi(\Omega)}+1\right)
\stackrel{(\ref{def_psi}),(\ref{cijpm})}{\underset{d}{\lesssim}}
$$
$$
\lesssim 1+\sum \limits _{i\in I_j^{\pm}}\frac{2^{m\pm
1}n\Psi(\{v_i^{m\pm 1,n}\})}{\Psi({\cal T})}
\stackrel{(\ref{prop_psi}),(\ref{ij_pm})}{\underset{d}{\lesssim}}
1+\frac{2^{m\pm 1}n \Psi({\cal T}_j^{m,n})}{\Psi({\cal
T})}\stackrel{(\ref{phi_e_phi})}{\underset{d}{\lesssim}} 1.
$$
Let $E\in T_{l_{j,m,n}}(\Delta_j^{m,n})$ for some $j\in
J_{m,n}$ and let $E$ overlap with at least two elements
of ${\cal B}_{n,m\pm 1}$. Then $\Delta_j^{m,n}=\Delta
_i^{m\pm 1,n}$ for some $i\in J_{m\pm 1,n}$ and
$$
{\rm card}\, \{E'\in {\cal B}_{n,m\pm 1}:\; ({\rm int}\, E')\cap
({\rm int}\, E)\ne \varnothing \}=
$$
$$
={\rm card}\, \{E'\in T_{l_{i,m\pm 1,n}}(\Delta _i^{m\pm 1,n}):\;
({\rm int}\, E')\cap ({\rm int}\, E)\ne \varnothing\}=
$$
$$
={\rm card}\, \{E'\in T_{l_{i,m\pm 1,n}}(\Delta _j^{m,n}):\; ({\rm
int}\, E')\cap ({\rm int}\, E)\ne \varnothing\}\le C(d)
$$
by item 4 of Lemma \ref{partition}, by the definition of
$l_{j,m,n}$ and $l_{i,m\pm 1,n}$ and by the inequality $\lceil
2a\rceil\le 2\lceil a\rceil$.
\end{proof}

\section{The spline approximation and the estimate of widths}
In the papers \cite{resh1, resh2} the integral representation for
smooth functions defined on a John domain in terms of their $r$-th
derivatives was obtained. Here we shall formulate this result for
functions that vanish on some ball and then we shall repeat more
accurately some steps of the proof from \cite{resh1, resh2}.

In \cite{resh1} the following equivalent definition of a John
domain was given.
\begin{Def}
\label{fca1} A domain $\Omega$ satisfies the John condition if
there exist $0<\rho<R$ and $x_*\in \Omega$ such that for any $x\in
\Omega$ there exists a curve $\gamma_x$ with properties 1 and 2
from Definition \ref{fca} such that $T(x)\le R$ and for any $t\in
[0, \, T(x)]$ the following inequality holds:
\begin{align}
\label{dgx} {\rm dist}\, (\gamma_x(t), \, \partial \Omega)\ge
\frac{\rho}{T(x)}t.
\end{align}
\end{Def}
Let us check the equivalence of Definitions \ref{fca} and
\ref{fca1}. Indeed, if the domain $\Omega$ satisfies the John
condition in the sense of Definition \ref{fca1}, then $\Omega \in
{\bf FC}(a)$ for $a<\frac{\rho}{R}$. Conversely, let $\Omega \in
{\bf FC}(a)$. Without loss of generality we may assume that $a<1$.
Let $R_0={\rm dist}\, (x_*, \, \partial \Omega)$. Then for any
$x\in \Omega$ we have $a \cdot T(x)\le R_0$. Set
\begin{align}
\label{rr0a2} R=\frac{R_0}{a}, \;\; \rho=\frac{aR_0}{2}.
\end{align}
Check (\ref{dgx}). If $T(x)\ge \frac{R_0}{2}$, then
$$
{\rm dist}\, (\gamma_x(t), \, \partial \Omega)\ge at=\frac{2\rho}{R_0}t\ge
\frac{\rho}{T(x)}t.
$$
If $T(x)<\frac{R_0}{2}$, then $|\gamma_x(t)-x_*|<\frac{R_0}{2}$
for any $t\in [0, \, T(x)]$. Hence,
$$
{\rm dist}\, (\gamma_x(t), \, \partial \Omega)\ge \frac{R_0}{2}\ge
\frac{aR_0}{2}\cdot \frac{t}{T(x)}= \frac{\rho}{T(x)}t.
$$
Notice that
\begin{align}
\label{rrho} \frac{\rho}{R}=\frac{a^2}{2}.
\end{align}

\begin{trma}
\label{reshteor} Let $\Omega\in {\bf FC}(a)$, let the point $x_*$
and the curves $\gamma_x$ be such as in Definition \ref{fca},
$R_0={\rm dist}\, (x_*, \, \partial \Omega)$, $r\in \N$. Then
there exist measurable functions $H_{\overline{\beta}}:
\Omega\times \Omega\rightarrow \R$, $\overline{\beta}=(\beta_1, \,
\dots, \, \beta_d)\in \Z_+^d$, $|\overline{\beta}|=r$, such that
the inclusion ${\rm supp}\, H_{\overline{\beta}}(x, \,
\cdot)\subset \cup _{t\in [0, \, T(x)]}B_{at}(\gamma_x(t))$ and
the inequality $|H_{\overline{\beta}}(x, \,
y)|\underset{a,d,r}{\lesssim}|x-y|^{r-d}$ hold for any $x\in
\Omega$, and for any function $f\in C^\infty(\Omega)$,
$f|_{B_{R_0/2}(x_*)}=0$ the following representation holds:
$$
f(x)=\sum \limits _{|\overline{\beta}|=r}\int \limits _\Omega
H_{\overline{\beta}}(x, \, y) \nabla ^{\overline{\beta}} f(y) \,
dy.
$$
\end{trma}
\begin{proof}
Without loss of generality we may assume that $a<1$. From formulas
(5.10), (5.14), (5.20) and (5.22) of the paper \cite{resh1}
follows that
$$
f(x)=\int \limits _{B_{\hat \rho}(x_*)}f(y)\theta(y-x_*)\, dy
+\sum \limits _{j=1}^d\int \limits _\Omega H_j(x, \, y)
\frac{\partial f}{\partial y_j}(y) \, dy,
$$
where $\theta(\cdot)$ is some smooth function whose support is
contained in the ball $B_{\hat \rho}(0)$, $\hat \rho=\frac{\rho}
{2(1+\sqrt{d})}\stackrel{(\ref{rr0a2})}{=}\frac{aR_0}{4(1+\sqrt{d})}\le
\frac{R_0}{2}$, $H_j(x, \, y)=\sum \limits _{\nu\in \N}
\omega_\nu(x)H_{j,\nu}(x, \, y)$, $\omega_\nu$ is a smooth
partition of unity, and measurable functions $H_{j,\nu}$ are
represented as
$$
H_{j,\nu}(x, \, y)=\int \limits _0^{T(x)}\psi_{j,\nu}(x, \, y, \,
t)\theta\left(R\frac{y-\gamma_x(t)}{t}\right)\, dt
$$
(the functions $\psi_{j,\nu}$ are defined in the formula (5.14));
here
$$
H_j(x, \, y)\underset{d}{\lesssim}
\left(\frac{R}{\rho}\right)^d|x-y|^{1-d}
\stackrel{(\ref{rrho})}{=} \left(\frac{2}{a^2}\right)^d
|x-y|^{1-d}.
$$
Since $\hat \rho\le \frac{R_0}{2}$, then for any smooth function
$f$ such that $f|_{B_{R_0/2}(x_*)}=0$ we get
\begin{align}
\label{fxsum} f(x)=\sum \limits _{j=1}^d\int \limits _\Omega
H_j(x, \, y) \frac{\partial f}{\partial y_j}(y) \, dy.
\end{align}
If $|y-\gamma_x(t)|\ge at$ for any $t\in [0, \, T(x)]$, then the
assumption $a<1$ implies that
$$
|y-\gamma_x(t)|\ge \frac{a^2}{4(1+\sqrt{d})}t
\stackrel{(\ref{rrho})}{=} \frac{\hat \rho}{R}t;
$$
since ${\rm supp}\, \theta \subset B_{\hat \rho}(0)$, then
$H_{j,\nu}(x, \, y)=0$. Hence, ${\rm supp}\, H_j(x, \,
\cdot)\subset \cup _{t\in [0, \, T(x)]}B_{at}(\gamma_x(t))$.

Let us prove the theorem for arbitrary $r\in \N$ following the
arguments from \cite{resh2}. Let
$$
\varphi(x, \, \xi)=\sum \limits _{|\overline{\beta}|\le
r-1}\frac{(x-\xi)^{\overline{\beta}}}{\overline{\beta}!}\nabla
^{\overline{\beta}} f(\xi).
$$
Then for $f|_{B_{R_0/2}(x_*)}=0$ we have $\varphi(x, \,
\cdot)|_{B_{R_0/2}(x_*)} =0$, and (\ref{fxsum}) implies
$$
\varphi(x, \, \xi)=\sum \limits _{j=1}^d\int \limits _\Omega
H_j(\xi, \, y) \frac{\partial \varphi}{\partial y_j}(x, \, y) \,
dy.
$$
Taking $\xi=x$, we get
$$
f(x)=\sum \limits _{j=1}^d\int \limits _\Omega H_j(x, \, y)
\frac{\partial \varphi}{\partial y_j}(x, \, y) \, dy.
$$
It remains to apply the formula (5) from the paper \cite{resh2}:
$$
\frac{\partial \varphi}{\partial y_j}(x, \, y)=\sum \limits
_{|\overline{\beta}|=r-1}
\frac{(x-y)^{\overline{\beta}}}{\overline{\beta}!}\nabla
^{\overline{\beta}+\delta_j}f(y),
$$
where $\delta _j=(\delta _{j,1}, \, \dots, \, \delta_{j,d})$ and
$\delta_{j,i}$ is the Kronecker symbol.
\end{proof}
The following theorem is proved in \cite{adams, adams1}; see also
\cite{leoni} (page 566) and \cite{mazya1} (page 51).
\begin{trma}
\label{adams_etc} Let $1<\tilde p<\tilde q<\infty$, $d\in \N$,
$r>0$, $\frac rd+\frac{1}{\tilde q}-\frac{1}{\tilde p}=0$,
$$
Tf(x)=\int \limits _{\R^d} f(y)|x-y|^{r-d}\, dy.
$$
Then the operator $T:L_{\tilde p}(\R^d)\rightarrow L_{\tilde
q}(\R^d)$ is bounded.
\end{trma}
\begin{Lem}
\label{emb_mon} Let $r$, $d\in \N$, let $\Omega\in {\bf FC}(a)$ be
a bounded domain in $\R^d$, let $1<\tilde p\le \infty$, $1\le
\tilde q<\infty$, $\frac{1}{\tilde \varkappa}=\frac
rd+\frac{1}{\tilde q}-\frac{1}{\tilde p}\ge 0$, let the functions
$\tilde g$, $\tilde v:\Omega\rightarrow \R_+$ satisfy the
conditions of Theorem \ref{main}. Let ${\cal T}$, $F:{\bf V}({\cal
T})\rightarrow \Theta(\Omega)$ be the tree and the mapping defined
in Lemma \ref{constr_omega_t}, let $\tilde{\cal T}$ be a subtree
in ${\cal T}$, $\tilde \Omega=\Omega_{\tilde{\cal T},F}$. Then for
any function $f\in W^r_{\tilde p,\tilde g}(\Omega)$ there exists a
polynomial $P_f$ of degree not exceeding $r-1$ such that
\begin{align}
\label{fmpf_lqv} \|f-P_f\|_{L_{\tilde q,\tilde v}(\tilde\Omega)}
\underset{\tilde p,\tilde q,r,d,a,c_0}{\lesssim} \|\tilde g\tilde
v\|_{L_{\tilde \varkappa}(\tilde\Omega)} \left\| \frac{\nabla
^rf}{\tilde g}\right\|_{L_{\tilde p}(\tilde\Omega)}.
\end{align}
Here the mapping $f\mapsto P_f$ is linear.
\end{Lem}
\begin{proof}
By Corollary \ref{cor_omega_t}, we have $\tilde \Omega\in {\bf
FC}(b_*)$ with $b_*=b_*(a, \, d)>0$. Let $\gamma_x:[0, \,
T(x)]\rightarrow \tilde \Omega$ be the curve from Definition
\ref{fca}, $\gamma_x(T(x))=x_*$. By Lemma \ref{omega_t_cone},
$\gamma_x$ can be chosen so that
\begin{align}
\label{le} \cup _{t\in [0, \, T(x)]}B_{b_*t}(\gamma_x(t))\subset
\tilde\Omega_{\le F(w)}, \;\; \text{ if }x\in F(w).
\end{align}

The set $C^\infty(\Omega)\cap W^r_{\tilde p,\tilde g}(\Omega)$ is
dense in $W^r_{\tilde p,\tilde g}(\Omega)$ (it can be proved
similarly as in the non-weighted case, see, e.g., \cite{mazya1},
page 16).\footnote{Here $C^\infty(\Omega)$ is the space of
functions that are smooth on the open set $\Omega$ yet not
necessarily extendable to smooth functions on the whole space
$\R^d$.} Hence, it is sufficiently to check (\ref{fmpf_lqv}) for
smooth functions.

{\bf Step 1.} Let $R_0={\rm dist}\, (x_*, \, \partial
\tilde\Omega)$. Prove that (\ref{fmpf_lqv}) follows from the
estimate
\begin{align}
\label{fmpf_lqv1} \|f\|_{L_{\tilde q,\tilde v}(\tilde\Omega)}
\underset{\tilde p,\tilde q,r,d,a,c_0}{\lesssim} \|\tilde g\tilde
v\|_{L_{\tilde \varkappa}(\tilde\Omega)} \left\| \frac{\nabla
^rf}{\tilde g}\right\|_{L_{\tilde p}(\tilde\Omega)}, \;\; f\in
C^\infty(\Omega), \;\; f|_{B_{R_0/2}(x_*)}=0.
\end{align}
Indeed, it was proved in \cite{sobolev1} (see also \cite{mazya1})
that for any function $f\in C^\infty(\Omega)$ there exists a
polynomial $P_f$ of degree not exceeding $r-1$ such that
\begin{align}
\label{nab_k_fpf} \|\nabla^k(f-P_f)\|_{L_{\tilde
q}(B_{3R_0/4}(x_*))} \underset{\tilde p,\tilde q,r,d}{\lesssim}
R_0^{r-k+\frac{d}{\tilde q}-\frac{d}{\tilde p}} \|\nabla ^r f\|
_{L_{\tilde p}(B_{3R_0/4}(x_*))},\;\; 0\le k\le r-1,
\end{align}
and the mapping $f\mapsto P_f$ is linear. Let
$\psi_0:\R^d\rightarrow [0, \, 1]$, $\psi_0\in C^\infty(\R^d)$,
${\rm supp}\, \psi_0\subset B_{3/4}(0)$, $\psi_0|_{B_{1/2}(0)}=1$,
$\psi(x)=\psi_0\left(\frac{x-x_*}{R_0}\right)$. Then
\begin{align}
\label{psi_fpf} \|\psi(f-P_f)\|_{L_{\tilde q}(\tilde\Omega)}\le
\|f-P_f\|_{L_{\tilde q}(B_{3R_0/4}(x_*))} \stackrel
{(\ref{nab_k_fpf})}{\underset{\tilde p,\tilde q,r,d}{\lesssim}}
R_0^{r+\frac{d}{\tilde q}-\frac{d}{\tilde p}} \|\nabla ^r
f\|_{L_{\tilde p} (B_{3R_0/4}(x_*))}.
\end{align}
From (\ref{tg_dg}) and (\ref{delta2}) follows that
\begin{align}
\label{ap_const} \frac{\tilde g(x)}{\tilde g(y)}\underset{c_0, \,
d}{\asymp} 1, \;\; \frac{\tilde v(x)}{\tilde v(y)}\underset{c_0,
\, d}{\asymp} 1, \;\; x, \, y\in B_{3R_0/4}(x_*).
\end{align}
Therefore,
$$
\|f-P_f\|_{L_{\tilde q,\tilde v}(\tilde \Omega)}\le
\|\psi(f-P_f)\|_{L_{\tilde q,\tilde v}(\tilde\Omega)}+
\|(1-\psi)(f-P_f)\|_{L_{\tilde q,\tilde v}(\tilde\Omega)}
\stackrel{(\ref{fmpf_lqv1}), (\ref{psi_fpf}),
(\ref{ap_const})}{\underset{\tilde p,\tilde
q,r,d,a,c_0}{\lesssim}}
$$
$$
\lesssim \|\tilde g\tilde v\|_{L_{\tilde
\varkappa}(B_{3R_0/4}(x_*))} \left\|\frac{\nabla ^r f}{\tilde
g}\right\| _{L_{\tilde p}(B_{3R_0/4}(x_*))} +\|\tilde g\tilde v\|
_{L_{\tilde \varkappa}(\tilde \Omega)} \left\|\frac{\nabla ^r
[(1-\psi)(f-P_f)]}{\tilde g}\right\| _{L_{\tilde p}(\tilde
\Omega)}\le
$$
$$
\le \|\tilde g\tilde v\| _{L_{\tilde \varkappa}(\tilde
\Omega)}\left(2\left\|\frac{\nabla ^r f}{\tilde g}\right\|
_{L_{\tilde p}(\tilde \Omega)} + \left\|\frac{\nabla ^r
[\psi(f-P_f)]}{\tilde g}\right\| _{L_{\tilde
p}(B_{3R_0/4}(x_*))}\right) \underset{\tilde
p,r,d,\psi_0}{\lesssim}
$$
$$
\lesssim \|\tilde g\tilde v\| _{L_{\tilde \varkappa}(\tilde
\Omega)}\left(2\left\|\frac{\nabla ^r f}{\tilde g}\right\|
_{L_{\tilde p}(\tilde \Omega)} + \sum \limits _{k=0}^r
R_0^{k-r}\left\|\frac{\nabla ^k(f-P_f)}{\tilde g}\right\|
_{L_{\tilde p} (B_{\frac{3R_0}{4}}(x_*))}\right)
\stackrel{(\ref{nab_k_fpf}) , (\ref{ap_const})}{\underset{\tilde
p,r,d,c_0}{\lesssim}}\|\tilde g\tilde v\| _{L_{\tilde
\varkappa}(\tilde \Omega)}\left\|\frac{\nabla ^r f}{\tilde
g}\right\| _{L_{\tilde p}(\tilde \Omega)}.
$$

{\bf Step 2.} Let $ f\in C^\infty(\Omega)$,
$f|_{B_{R_0/2}(x_*)}=0$, $\varphi(x)=\frac{|\nabla ^r
f(x)|}{\tilde g(x)}$. From Theorem \ref{reshteor} follows that for
any $x\in \tilde\Omega$ there exists a set $G_x\subset \cup _{t\in
[0, \, T(x)]}B_{b_*t}(\gamma_x(t))$ such that $\{(x, \, y)\in
\tilde\Omega\times \tilde\Omega:\; y\in G_x\}$ is measurable and
$$
|f(x)|\underset{r,d,a}{\lesssim} \int \limits
_{G_x}|x-y|^{r-d}\tilde g(y)\varphi(y)\, dy.
$$
Hence, in order to prove (\ref{fmpf_lqv1}), it is sufficient to
obtain the estimate
\begin{align}
\label{est_co_op} \left(\int \limits_{\tilde\Omega} \tilde
v^{\tilde q}(x)\left(\int \limits_{G_x}|x-y|^{r-d}\tilde
g(y)\varphi(y)\, dy\right)^{\tilde q}\, dx\right)^{1/\tilde q}
\underset{\tilde p, \tilde q,r,d,c_0,a}{\lesssim} \|\tilde g\tilde
v\|_{L_{\tilde \varkappa}(\tilde\Omega)}\|\varphi\|_{L_{\tilde
p}(\tilde\Omega)}.
\end{align}

{\bf Step 3.} Let $\frac{1}{\tilde \varkappa}>0$. Denote by $w_*$
the minimal vertex of the tree $\tilde {\cal T}$. Let ${\cal T}'$,
${\cal T}''$ be subtrees in $\tilde {\cal T}$ rooted at $w_*$, and
let $E'=\tilde\Omega\backslash \Omega_{{\cal T}',F}$,
$E''=\Omega_{{\cal T}'',F}$. Show that (\ref{est_co_op}) holds for
$\tilde g=\chi_{E'}$, $\tilde v=\chi_{E''}$. Actually, $E'\cap
E''=\sqcup _i E_i$, where $E_i=\Omega_{{\cal A}_i,F}$, ${\cal
A}_i\subset {\cal T}$ are subtrees with pairwise incomparable
minimal vertices. From (\ref{le}) follows that for any $x\in
\Omega_{{\cal T}',F}$ the inclusion $G_x\subset \Omega_{{\cal
T}',F}$ holds and $G_x\cap E'=\varnothing$, and for any $x\in E_i$
the inclusion $G_x\cap E'\subset E_i$ holds. Hence, the left-hand
side of (\ref{est_co_op}) does not exceed
$$
\left(\sum \limits _i\int \limits_{E_i} \left(\int
\limits_{E_i}|x-y|^{r-d}\varphi(y)\, dy\right)^{\tilde q}\,
dx\right)^{1/\tilde q}=:M.
$$
By Corollary \ref{cor_omega_t}, $E_i\in {\bf FC}(b_*)$, which
implies $({\rm diam}\, E_i)^d\underset{a,d}{\lesssim} {\rm mes}\,
E_i$. Applying Theorem \ref{adams_etc} and the H\"{o}lder
inequality, we get
$$
M\underset{\tilde p,\tilde q,r,d,a}{\lesssim} \left(\sum \limits_i
({\rm mes}\, E_i)^{\frac{\tilde q}{\tilde
\varkappa}}\|\varphi\|^{\tilde q}_{L_{\tilde
p}(E_i)}\right)^{\frac{1}{\tilde q}}\le ({\rm mes}\, (E'\cap
E''))^{1/\tilde \varkappa}\|\varphi\|_{L_{\tilde
p}(\tilde\Omega)}=\|\tilde g\tilde v\|_{L_{\tilde
\varkappa}(\tilde\Omega)}\|\varphi\|_{L_{\tilde p}(\tilde\Omega)}
$$
(for $\tilde p\le \tilde q$ the second inequality follows from
$\left(\sum \limits_i a_i^{\tilde q}\right)^{\frac{1}{\tilde
q}}\le \left(\sum \limits_i a_i^{\tilde p}\right)^{\frac{1}{\tilde
p}}$, and for $\tilde p>\tilde q$ it follows from the H\"{o}lder
inequality and from $\left(\sum \limits_i a_i^{\frac{\tilde
p\tilde q}{\tilde \varkappa(\tilde p-\tilde
q)}}\right)^{\frac{1}{\tilde q}-\frac{1}{\tilde p}}\le \left(\sum
\limits_i a_i\right)^{\frac{1}{\tilde \varkappa}}$).

{\bf Step 4.} Let $\frac{1}{\tilde \varkappa}>0$. Show that there
exist two sequences of subtrees $\{{\cal T}_j'\}_{j\in \Z_+}$ and
$\{{\cal T}''_j\}_{j\in \Z_+}$ in the tree ${\cal T}$ and there
are functions $\hat g$, $\hat v:\Omega\rightarrow \R_+$ with the
following properties:
\begin{enumerate}
\item ${\cal T}'_0={\cal T}''_0=\varnothing$, ${\cal T}'_j
\subset {\cal T}'_{j+1}$, ${\cal T}''_j \subset {\cal T}''_{j+1}$,
$j\in \Z_+$, $\cup_{j\in \Z_+}{\cal T}'_j=\cup_{j\in \Z_+} {\cal
T}''_j={\cal T}$;
\item $\hat g|_{\Omega_{{\cal T}'_j,F}\backslash \Omega
_{{\cal T}'_{j-1},F}}=C'_j$, $\hat v|_{\Omega_{{\cal
T}''_j,F}\backslash \Omega_{{\cal T}''_{j-1},F}}=C''_j$, $j\in
\N$;
\item the sequence $\{C'_j\}_{j\in \N}$ increases and the sequence
$\{C''_j\}_{j\in \N}$ decreases;
\item $\tilde g(x)\underset{a,d,c_0}{\asymp}\hat g(x)$, $\tilde v(x)
\underset{a,d,c_0}{\asymp} \hat v(x)$.
\end{enumerate}
Let us construct the function $\hat g$ (the function $\hat v$ can
be constructed similarly). Let $\Gamma '\subset
\partial \Omega$ be the set from the conditions of Theorem \ref{main}, let ${\cal T}'$
be a subtree in ${\cal T}$, let $w'$ be the minimal vertex of
${\cal T}'$, $m\in \Z$, ${\rm dist}\,(F(w'), \, \Gamma')\in
[2^{-m}, \, 2^{-m+1})$. Denote by ${\cal S}_{{\cal T}'}$ the
maximal tree in the sense of inclusions from the set
of trees ${\cal S}'\subset {\cal T}'$ rooted at $w'$ and
satisfying
\begin{align}
\label{dist_fw} {\rm dist}\, (F(w), \, \Gamma')\ge 2^{-m}, \;\;
w\in {\bf V}({\cal S}').
\end{align}
Show that for any $w\in {\bf V}({\cal S}_{{\cal T}'})$, $x\in
F(w)$ we have
\begin{align}
\label{dist_fw_up} {\rm dist}\, (x, \,
\Gamma')\underset{a,d}{\lesssim} 2^{-m}.
\end{align}
Indeed, choose $x'\in F(w')$ such that ${\rm dist}\, (x', \,
\Gamma')<2^{-m+1}$. Since $\Omega_{{\cal S}_{{\cal T}'},F}\in {\bf
FC}(b_*(a, \, d))$ (see Corollary \ref{cor_omega_t}), then
Definition \ref{fca} yields that
$|x-x'|\underset{a,d}{\lesssim}2^{-{\bf m}(F(w'))}$. By Theorem
\ref{whitney}, $2^{-{\bf m}(F(w'))}\underset{d}{\asymp} {\rm
dist}\, (x', \, \partial \Omega)$. Therefore,
$$
{\rm dist}\, (x, \, \Gamma')\le |x-x'|+{\rm dist}\, (x', \,
\Gamma')\underset{a,d}{\lesssim} {\rm dist}\, (x', \,
\partial \Omega)+2^{-m+1}\le {\rm dist}\, (x', \,
\Gamma')+2^{-m+1}\le 2^{-m+2}.
$$

The trees ${\cal T}'_j$ are constructed by induction on $j\in
\Z_+$. Set ${\cal T}'_0=\varnothing$. Let the trees  ${\cal T}'_i$
be constructed for $i\in \{0, \, \dots, \, j\}$, and let the
numbers $C'_i=\varphi_{\tilde g}(2^{-m_i})$ be defined (see
(\ref{tg_dg})), where $m_i\in \Z$, $i\in \{1, \, \dots, \, j\}$,
$m_1\le \dots \le m_j$. In addition, suppose that ${\cal T}={\cal
T}'_j\sqcup \left(\sqcup _{s=1}^{s_0(j)}{\cal T}'_{j,s}\right)$,
where ${\cal T}'_{j,s}$ are trees rooted at $w_{j,s}$, $s_0(j)\in
\N\cup \{\infty\}$, $w_{j,s}$ are adjacent to some vertices of
${\cal T}'_j$ and
$$
\dist \, (F(w_{j,s}), \, \Gamma')\in [2^{-m_{j,s}}, \,
2^{-m_{j,s}+1}), \;\; m_{j,s}\in \Z, \;\; m_{j,s}\ge m_j+1.
$$
Set $m_{j+1}=\min _{1\le s\le s_0(j)}m_{j,s}$,
$C'_{j+1}=\varphi_{\tilde g}(2^{-m_{j+1}})$, $I_j=\left\{s\in
\overline{1, \, s_0(j)}:\; m_{j,s}=m_{j+1}\right\}$, ${\cal
T}'_{j+1}={\cal T}'_j\cup \left(\cup _{s\in I_j}{\cal S}_{{\cal
T}'_{j,s}}\right)$, $\hat g|_{\Omega_{{\cal T}'_{j+1},F}\backslash
\Omega_{{\cal T}'_j,F}}=C'_{j+1}$. Then the properties 1 and 2
hold by the construction, the property 3 holds since the function
$\varphi_{\tilde g}$ decreases and the sequence $\{m_j\}_{j\in
\N}$ increases. The property 4 follows from (\ref{tg_dg}),
(\ref{delta2}), (\ref{dist_fw}) and (\ref{dist_fw_up}).

{\bf Step 5.} Let us prove (\ref{est_co_op}). If $\frac{1}{\tilde
\varkappa}=0$, then it follows from Theorem \ref{adams_etc}
(remind that by the condition of Theorem \ref{main} in this case
we have $\tilde g=1$ and $\tilde v=1$). Let $\frac{1}{\tilde
\varkappa}>0$. We may assume that $\tilde g=\hat g$, $\tilde
v=\hat v$, where $\hat g$ and $\hat v$ are functions constructed
at step 4. Applying the estimate which is obtained at step 3, we
argue similarly as in the paper \cite{avas2} (see Lemma 5.4 on the
page 487). Notice that in the case $\tilde q=1$ the corresponding
set $\tilde G_y$ is defined as $\tilde G_y=\{x\in \tilde\Omega:\;
y\in G_x\}$. If $y\in F(w)$, $w\in {\bf V}(\tilde{\cal T})$, then
$\tilde G_y\subset \Omega_{\tilde{\cal T}_w,F}$.
\end{proof}
Denote by ${\cal P}_{r-1}(\R^d)$ the space of polynomials on
$\R^d$ of degree not exceeding $r-1$. For a measurable set
$E\subset \R^d$ set ${\cal P}_{r-1}(E)= \{f|_E:\, f\in {\cal
P}_{r-1}(\R^d)\}$.

Let $G\subset \R^d$ be a domain and let $T=\{\Omega
_i\}_{i=1}^{i_0}$ be its finite partition. Denote
\begin{align}
\label{calsrtgsgr} {\cal S}_{r,T}(G)=\{S:G\rightarrow \R:
S|_{\Omega _i}\in {\cal P}_{r-1}(\Omega _i), \; 1\le i \le i_0\};
\end{align}
for $f\in L_{q,v}(G)$ set
\begin{align}
\label{norm_f_pqtv} \| f\| _{p,q,T,v}= \left(\sum \limits
_{i=1}^{i_0} \| f\| _{L_{q,v}(\Omega _i)}^{\sigma_{p,q}}\right)
^{\frac{1}{\sigma_{p,q}}},
\end{align}
where $\sigma _{p,q}=\min \{p, \, q\}$. Denote by $L_{p,q,T,v}(G)$
the space of functions $f\in L_{q,v}(G)$ with the norm $\| \cdot\|
_{p,q,T,v}$. Notice that $\| f\| _{p,q,T,v}\ge \| f\|
_{L_{q,v}(G)}$.

\renewcommand{\proofname}{\bf Proof of Theorem \ref{main}}
\begin{proof}
The lower estimate can be proved similarly as in \cite{vas_mnogo}.
In order to obtain the upper estimate, we shall prove that for any
$\varepsilon>0$ there exists $N(\varepsilon)\in \N$ such that for
any $n\in \N$, $n\ge N(\varepsilon)$, $m\in \Z_+$ there exists a
partition $\hat T_{m,n,\varepsilon}=\hat
T_{m,n,\varepsilon}(\Omega)=\{G^{m,n,\varepsilon}_j\}_{j=1}^{\nu_{m,n}}$
of $\Omega$ with the following properties:
\begin{enumerate}
\item $\nu_{m,n}\underset{d}{\lesssim}2^mn$;
\item for any function $f\in W^r_{p,g}(\Omega)$ there exists a spline
$\hat S_{m,n,\varepsilon}(f)\in {\cal S}_{r,\hat
T_{m,n,\varepsilon}}(\Omega)$ such that
\begin{align}
\label{appr_spl1} \|f-\hat S_{m,n,\varepsilon}(f)\|_{p,q,\hat
T_{m,n,\varepsilon},v}\underset{a,d,p,q,r,c_0,\alpha,\beta}{\lesssim}
(\|gv\|_\varkappa+\varepsilon) (2^mn)^{-\frac rd+\left(\frac
1p-\frac 1q\right)_+},
\end{align}
and the mapping $f\mapsto \hat S_{m,n,\varepsilon}(f)$ is linear;
\item for any $G^{m,n,\varepsilon}_j$
$${\rm card}\, \{i\in \{1, \, \dots, \, \nu_{m\pm 1,n}\}:\; {\rm mes}\, (G^{m,n,\varepsilon}_j
\cap G^{m\pm 1,n,\varepsilon}_i)>0\}\underset{d}{\lesssim}1.$$
\end{enumerate}
Then by repeating arguments from the paper \cite{avas2} (see pages
499--501), we get the desired upper estimate for widths.

{\bf Step 1.} Let us consider the case $\alpha<\infty$ and
$\beta<\infty$ only (if $\alpha=\infty$ or $\beta=\infty$, then
arguments are similar with slight changes in the definition of the
function $\Phi$). Let $\mu_1(E)=\int \limits _E g_0^\alpha(x)\,
dx$, $\mu_2(E)=\int \limits _E v_0^\beta(x)\, dx$. If
$\frac{1}{\tilde \varkappa}:=\frac rd+\frac 1q-\frac
1p-\frac{1}{\alpha}-\frac{1}{\beta}>0$, then we set $l_*=3$,
$\mu_3(E)=\int \limits_E \tilde g^{\tilde \varkappa}(x)\tilde
v^{\tilde \varkappa}(x)\, dx$,
\begin{align}
\label{alphaj} \alpha _1=\frac{\frac{1}{\alpha}}{\frac rd+\frac
1q-\frac 1p}, \;\; \alpha_2= \frac{\frac{1}{\beta}}{\frac rd+\frac
1q-\frac 1p}, \;\; \alpha_3= \frac{\frac rd+\frac 1q-\frac
1p-\frac{1}{\alpha}-\frac{1}{\beta}} {\frac rd+\frac 1q-\frac 1p};
\end{align}
if $\frac rd+\frac 1q-\frac
1p-\frac{1}{\alpha}-\frac{1}{\beta}=0$, then we set $l_*=2$ and
and define $\alpha_1$ and $\alpha_2$ by the formula
(\ref{alphaj}). Define the function $\Phi$ by (\ref{def_phi}). In
addition, set $\frac{1}{\tilde p}=\frac 1p+\frac{1}{\alpha}$,
$\frac{1}{\tilde q}=\frac{1}{q}-\frac{1}{\beta}$. From conditions
of the Theorem follows that $\tilde p>1$ and $\tilde q<\infty$.

{\bf Step 2.} Let $({\cal T}, \, w_*)$ and $F$ be the tree and the
mapping defined in Lemma \ref{constr_omega_t}. For $k\in \N$ we
denote by ${\cal T}_{\le k}$ a subtree in ${\cal T}$ such that
${\bf V}({\cal T}_{\le k})=\{w\in {\bf V}({\cal T}):\; \rho(w_*,
\, w)\le k\}$. Since ${\rm card}\, {\bf V}_1(w)<\infty$ for any
$w\in {\bf V}({\cal T})$, then the set ${\bf V}({\cal T}_{\le k})$
is finite.

Fix $\delta>0$ and choose $k\in \N$ such that
\begin{align}
\label{k_def} \delta_k:=\Phi(\Omega\backslash \Omega _{{\cal
T}_{\le k},F})\le \delta.
\end{align}
Then ${\cal T}={\cal T}_{\le
k}\sqcup\left(\sqcup_{l=1}^{l_0(k)}{\cal T}_{k,l}\right)$, where
${\cal T}_{k,l}$ are trees rooted at $\hat w_{k,l}$, $l_0(k)\in
\N$. Set $\delta_{k,l}=\Phi\left(\Omega_{{\cal
T}_{k,l},F}\right)$.

{\bf Step 3.} If a domain $U$ is a finite union of non-overlapping
cubes\footnote{Here a cube is set that contains an open cube and is
contained in closure of this cube.}, then for sufficiently
large $n$ the partition $\hat T_{m,n,\varepsilon}(U)$ can be constructed
in the same way as in \cite{vas_mnogo}. Define the partition $\hat
T_{m,n,\varepsilon/2;k}=\hat T_{m,n,\varepsilon/2}(\Omega_{{\cal
T}_{\le k},F})$ with properties similar to 1--3. In particular, for
any function $f\in W^r_{p,g}(\Omega)$ there exists a spline $\hat
S_{m,n,\varepsilon/2;k}(f)\in {\cal S}_{r,\hat
T_{m,n,\varepsilon/2;k}}(\Omega_{{\cal T}_{\le k},F})$ such that
\begin{align}
\label{pribl_kon_kub} \|f-\hat
S_{m,n,\varepsilon/2;k}(f)\|_{p,q,\hat
T_{m,n,\varepsilon/2;k},v}\underset{a,d,p,q,r,c_0,\alpha,\beta}{\lesssim}
\left(\|gv\|_\varkappa+\frac{\varepsilon}{2}\right) (2^mn)^{-\frac
rd+\left(\frac 1p-\frac 1q\right)_+},
\end{align}
and the mapping $f\mapsto \hat S_{m,n,\varepsilon/2;k}(f)$ is linear.

{\bf Step 4.} Let $n\ge l_0(k)$. For each $l\in \{1, \, \dots,
\, l_0(k)\}$ we set
\begin{align}
\label{de_nl} n_l=\left\{ \begin{array}{l} \left\lceil
n\frac{\delta_{k,l}}{\delta_k}\right\rceil, \text{ if } \delta_k>0, \\
1, \text{ if }\delta_k=0.
\end{array}\right.
\end{align}
Hence, if $\delta_k=0$, then $\sum \limits _{l=1}^{l_0(k)} n_l
=l_0(k)\le n$, and if $\delta_k>0$, then
\begin{align}
\label{sll1l0k} \sum \limits _{l=1}^{l_0(k)} n_l \le n\sum \limits
_{l=1}^{l_0(k)} \frac{\delta_{k,l}}
{\delta_k}+l_0(k)\stackrel{(\ref{pr1})}{\le} n+l_0(k)\le 2n.
\end{align}

{\bf Step 5.} Prove that for any $l\in \{1, \, \dots, \,
l_0(k)\}$, $n\ge l_0(k)$, $m\in \Z_+$ there exists a partition
$T_{m,n}^l=\{G^{m,n}_{j,l}\}_{j=1}^{\nu_{m,n,l}}$ of
$\Omega_{{\cal T}_{k,l},F}$ with the following properties: 1.
$\nu_{m,n,l}\underset{d}{\lesssim} 2^mn_l$, 2. for any function
$f\in W^r_{p,g}(\Omega)$ there exists a spline $S_{m,n,l}(f)\in {\cal
S}_{r, T_{m,n}^l}(\Omega_{{\cal T}_{k,l},F})$ such that
\begin{align}
\label{loc_est}
\|f-S_{m,n,l}(f)\|_{p,q,T_{m,n}^l,v}\underset{p,q,r,d,\alpha,\beta,a,c_0}{\lesssim}
\left(\frac{\delta}{2^mn}\right)^{\frac rd+\frac 1q-\frac
1p}\left(\sum \limits _{E\in T_{m,n}^l} \left\|\frac{\nabla ^r
f}{g}\right\|_{L_p(E)}^{\sigma_{p,q}}\right)^{\frac{1}{\sigma_{p,q}}},
\end{align}
and the mapping $f\mapsto S_{m,n,l}(f)$ is linear, 3. for any $G^{m,n}_{j,l}$
\begin{align}
\label{cardi1} {\rm card}\, \{i\in \{1, \, \dots, \, \nu_{m\pm
1,n,l}\}:\; {\rm mes}\, (G^{m,n}_{j,l} \cap G^{m\pm
1,n}_{i,l})>0\}\underset{d}{\lesssim}1.
\end{align}

Let ${\cal B}_{n_l,m}=\{E^{m,n_l}_{j,l}\}_{j=1}
^{\tilde\nu_{m,n,l}}$ be the partition of $\Omega_{{\cal
T}_{k,l},F}$ defined in Lemma \ref{approx}. From item 1 follows
that $\tilde\nu_{m,n,l}\underset{d}{\lesssim} 2^mn_l$. By item 2,
a), either $E^{m,n_l}_{j,l}\subset F(w)$ (with the strict
inclusion) and $E^{m,n_l}_{j,l}\in {\cal R}\cup \Xi(F(w))$ for
some $w=w_{j,l}^{m,n_l}\in {\bf V}({\cal T}_{k,l})$ (the set of
such $j$ will be denoted by $J^1_{m,n,l}$), or
$E^{m,n_l}_{j,l}=\Omega_{{\cal T}^{m,n_l}_{j,l},F}$ for some
subtree ${\cal T}^{m,n_l}_{j,l}\subset {\cal T}_{k,l}$ (the set of
such $j$ will be denoted by $J^2_{m,n,l}$). From item 2, b)
follows that
\begin{align}
\label{phigejmn} \Phi(E^{m,n_l}_{j,l})\underset{d}{\lesssim}
\frac{\Phi(\Omega_{{\cal T}_{k,l},F})}{2^mn_l}
\stackrel{(\ref{de_nl})}{\le} \frac{\delta_k}{2^mn}
\stackrel{(\ref{k_def})}{\le} \frac{\delta}{2^mn}.
\end{align}

Let $j\in J^1_{m,n,l}$. Then Theorem \ref{whitney}
together with (\ref{tg_dg}) and (\ref{delta2}) imply that
for any $x$, $y\in E_{j,l}^{m,n_l}$ we have $\frac{\tilde g(x)}{\tilde
g(y)}\underset{c_0,d}{\asymp}1$ and $\frac{\tilde v(x)}{\tilde
v(y)}\underset{c_0,d}{\asymp}1$. Hence, there exists a partition
$\Pi_{j,l}$ of the set $E_{j,l}^{m,n_l}$ into at most $2d$
measurable subsets with the following property: for any function
$f\in W^r_{p,g}(\Omega)$ there exists a spline $S_{j,l}(f)\in {\cal
S}_{r,\Pi_{j,l}}(E^{m,n_l}_{j,l})$ such that
\begin{align}
\label{2d}
\|f-S_{j,l}(f)\|_{p,q,\Pi_{j,l},v}\underset{p,q,r,d,\alpha,\beta,c_0}{\lesssim}
\Phi(E_{j,l}^{m,n_l})^{\frac{1}{\varkappa}} \left\|\frac{\nabla
^rf}{g}\right\|_{L_p(E_{j,l}^{m,n_l})}
\stackrel{(\ref{phigejmn})}{\underset{d,\varkappa}{\lesssim}}
\left(\frac{\delta}{2^mn}\right)^{\frac{1}{\varkappa}}
\left\|\frac{\nabla ^rf}{g}\right\|_{L_p(E_{j,l}^{m,n_l})}
\end{align}
(see the beginning of the proof of Theorem 3 in \cite{vas_mnogo}).

Let $j\in J^2_{m,n,l}$. By Lemma \ref{emb_mon}, there exists
a polynomial $P_{j,l}(f)$ of degree not exceeding $r-1$ such that
$$
\|f-P_{j,l}(f)\|_{L_{\tilde q,\tilde v}(E_{j,l}^{m,n_l})}
\underset{\tilde p,\tilde q,r,d,a,c_0}{\lesssim} \|\tilde g\tilde
v\|_{L_{\tilde \varkappa}(E_{j,l}^{m,n_l})}\left\|\frac{\nabla ^r
f}{\tilde g}\right\|_{L_{\tilde p}(E_{j,l}^{m,n_l})}.
$$
This together with the H\"{o}lder inequality yields
\begin{align}
\label{g_appr}
\begin{array}{c}
\|f-P_{j,l}(f)\|_{L_{q,v}(E_{j,l}^{m,n_l})}
\underset{p,q,\alpha,\beta,r,d,a,c_0}{\lesssim}
\|v_0\|_{L_\beta(E_{j,l}^{m,n_l})}\|\tilde g\tilde v\|_{L_{\tilde
\varkappa}(E_{j,l}^{m,n_l})}\|g_0\|_{L_\alpha(E_{j,l}^{m,n_l})}\left\|\frac{\nabla
^r f}{g}\right\|_{L_p(E_{j,l}^{m,n_l})}\stackrel{(\ref{alphaj})}{=}\\
=\Phi(E_{j,l}^{m,n_l})^{\frac rd+\frac 1q-\frac 1p}
\left\|\frac{\nabla ^r f}{g}\right\|_{L_p(E_{j,l}^{m,n_l})}
\stackrel{(\ref{phigejmn})}{\underset{d,\varkappa} {\lesssim}}
\left(\frac{\delta}{2^mn}\right)^{\frac rd+\frac 1q-\frac 1p}
\left\|\frac{\nabla ^rf}{g}\right\| _{L_p(E_{j,l}^{m,n_l})}.
\end{array}
\end{align}

Set $T_{m,n}^l=\left(\cup _{j\in J^1_{m,n,l}}\Pi_{j,l}\right) \cup
\{E_{j,l}^{m,n_l}\}_{j\in J^2_{m,n,l}}$,
$S_{m,n,l}(f)|_{E_{j,l}^{m,n_l}} =S_{j,l}(f)$, $j\in J^1_{m,n,l}$,
$S_{m,n,l}(f)|_{E_{j,l}^{m,n_l}}= P_{j,l}(f)$, $j\in J^2_{m,n,l}$.
Then the property 1 and (\ref{cardi1}) follow from items 1 and 3
of Lemma \ref{approx}; (\ref{2d}) and (\ref{g_appr}) imply
(\ref{loc_est}).

{\bf Step 6.} Put $\hat T_{m,n,\varepsilon}=\hat
T_{m,n,\varepsilon/2;k}\cup
\left(\cup_{l=1}^{l_0(k)}T^l_{m,n}\right)$,
$\hat S_{m,n,\varepsilon}(f)|_{\Omega_{{\cal T}_{\le k},F}}
=\hat S_{m,n,\varepsilon/2;k}(f)$, $\hat S_{m,n,\varepsilon}(f)|
_{\Omega_{{\cal T}_{k,l},F}}=S_{m,n,l}(f)$. The property 1 follows
from the estimate
\begin{align}
\label{aaaa} {\rm card}\, \hat T_{m,n,\varepsilon}={\rm
card}\,\hat T_{m,n,\varepsilon/2;k}+\sum \limits
_{l=1}^{l_0(k)}\nu_{m,n,l}\underset{d}{\lesssim} 2^mn+\sum \limits
_{l=1}^{l_0(k)}2^mn_l\stackrel{(\ref{sll1l0k})}{\underset{d}{\lesssim}}
2^mn.
\end{align}
The inequality (\ref{appr_spl1})
follows from (\ref{pribl_kon_kub}), (\ref{loc_est}) with
sufficiently small $\delta>0$, (\ref{aaaa}) and
the H\"{o}lder inequality. The property 3 of the partition $\hat
T_{m,n,\varepsilon}$ follows from the property 3 of the partitions $\hat
T_{m,n,\varepsilon/2;k}$ and $T^l_{m,n}$.
\end{proof}
In conclusion, the author expresses her sincere gratitude to
A.S. Kochurov for reading the manuscript and to O.V. Besov
for remarks on background.

\end{document}